\documentclass[12pt]{amsart}

\usepackage[all]{xy}
\usepackage{fullpage}
\usepackage[expansion=false]{microtype}
\usepackage{needspace}
\usepackage{latexsym}
\usepackage{amsmath}
\usepackage{amsfonts}
\usepackage{amssymb}
\usepackage{amsthm}
\usepackage{eucal}
\usepackage{enumerate}
\usepackage{mathrsfs}
\usepackage{graphicx}
\usepackage{graphics}
\usepackage{epstopdf}
\usepackage{amscd}
\usepackage{hyperref}
\usepackage{url}
\usepackage{color}
\usepackage{cancel}
\usepackage{enumerate}
\usepackage{amsmath,amsthm}
\usepackage{amssymb}
\usepackage{epsfig}
\usepackage{xy}
\usepackage{xypic}
\usepackage{tikz}

\usepackage{amsthm}
\usepackage{mathrsfs} 
\usepackage{amssymb}
\usepackage{tikz-cd}
\usepackage{hyperref}
\IfFileExists{stmaryrd.sty}{\usepackage{stmaryrd}}{%
  \newcommand{\sslash}{\mathbin{/\mkern-6mu/}}%
  \newcommand{\fatslash}{\mathbin{\boldsymbol{/}}}%
}
\usepackage[all]{xy}
\usepackage{todonotes}
\usepackage{cleveref}

\usepackage[all]{xy}
\usepackage{latexsym}
\usepackage{amsmath}
\usepackage{amsfonts}
\usepackage{amssymb}
\usepackage{amsthm}
\usepackage{eucal}
\usepackage{enumerate}
\usepackage{mathrsfs}
\usepackage{graphicx}
\usepackage{graphics}
\usepackage{epstopdf}
\usepackage{amscd}
\usepackage{hyperref}
\usepackage{url}
\usepackage{color}
\usepackage{cancel}
\usepackage{enumerate}
\usepackage{amsmath,amsthm}
\usepackage{amssymb}
\usepackage{epsfig}
\usepackage{xy}
\usepackage{xypic}
\usepackage{tikz}

\newtheorem{theorem}{Theorem}[subsection]
\newtheorem{corollary}[theorem]{Corollary}
\newtheorem{lemma}[theorem]{Lemma}
\newtheorem{proposition}[theorem]{Proposition}

\theoremstyle{remark}
\newtheorem{remark}[theorem]{Remark}

\theoremstyle{definition}

\theoremstyle{definition}

\newtheorem{example}[theorem]{Example}

\newcommand{\IC}{\mathrm{IC}}
\newcommand{\quash}[1]{}  

    \newcommand{\bR}{\mathbf{R}}
    \newcommand{\bC}{\mathbf{C}}
    \newcommand{\bZ}{\mathbf{Z}}

    \newcommand{\Hom}{\mathrm{Hom}}
    
    \newcommand{\Sym}{\operatorname{Sym}}
    
    \newcommand{\tildeboxtimes}{\widetilde{\boxtimes}}

    \newcommand{\Gr}{\mathrm{Gr}}
    
    \newcommand{\Sp}{\mathrm{Sp}}

  \newcommand{\fc}{\mathfrak{c}} 
   
 \newcommand{\fg}{\mathfrak{g}}   \newcommand{\fb}{\mathfrak{b}}  
 \newcommand{\fn}{\mathfrak{n}}  
    
    \newcommand{\fk}{\mathfrak{k}}

\newcommand{\on}{\mathrm}

\newcommand{\renfatslash}{{\hspace{-0.2em}\fatslash}}  

\tikzset{node distance=2em, ch/.style={circle,draw,on chain,inner sep=2pt},chj/.style={ch,join},every path/.style={shorten >=4pt,shorten <=4pt},line width=1pt,baseline=-1ex}

\DeclareMathSymbol{A}{\mathalpha}{operators}{`A}
\DeclareMathSymbol{B}{\mathalpha}{operators}{`B}
\DeclareMathSymbol{C}{\mathalpha}{operators}{`C}
\DeclareMathSymbol{D}{\mathalpha}{operators}{`D}
\DeclareMathSymbol{E}{\mathalpha}{operators}{`E}
\DeclareMathSymbol{F}{\mathalpha}{operators}{`F}
\DeclareMathSymbol{G}{\mathalpha}{operators}{`G}
\DeclareMathSymbol{H}{\mathalpha}{operators}{`H}
\DeclareMathSymbol{I}{\mathalpha}{operators}{`I}
\DeclareMathSymbol{J}{\mathalpha}{operators}{`J}
\DeclareMathSymbol{K}{\mathalpha}{operators}{`K}
\DeclareMathSymbol{L}{\mathalpha}{operators}{`L}
\DeclareMathSymbol{M}{\mathalpha}{operators}{`M}
\DeclareMathSymbol{N}{\mathalpha}{operators}{`N}
\DeclareMathSymbol{O}{\mathalpha}{operators}{`O}
\DeclareMathSymbol{P}{\mathalpha}{operators}{`P}
\DeclareMathSymbol{Q}{\mathalpha}{operators}{`Q}
\DeclareMathSymbol{R}{\mathalpha}{operators}{`R}
\DeclareMathSymbol{S}{\mathalpha}{operators}{`S}
\DeclareMathSymbol{T}{\mathalpha}{operators}{`T}
\DeclareMathSymbol{U}{\mathalpha}{operators}{`U}
\DeclareMathSymbol{V}{\mathalpha}{operators}{`V}
\DeclareMathSymbol{W}{\mathalpha}{operators}{`W}
\DeclareMathSymbol{X}{\mathalpha}{operators}{`X}
\DeclareMathSymbol{Y}{\mathalpha}{operators}{`Y}
\DeclareMathSymbol{Z}{\mathalpha}{operators}{`Z}

\newcommand{\ld}[1]{{\check{#1}}}
\newcommand{\GL}{\mathrm{GL}}

\newcommand{\SO}{\mathrm{SO}}
\newcommand{\GSO}{\mathrm{GSO}}

\title{Relative Langlands duality for linear periods}

\author{Tsao-Hsien Chen}
        \address{School of Mathematics, University of Minnesota, Twin cities, Minneapolis, MN 55455}
         \email{chenth@umn.edu}
         \author{Sanath K. Devalapurkar}
        \address{School of Mathematics. Institute for Advanced Study 1 Einstein Drive Princeton, New Jersey 08540}
         \email{skdevalapurkar@gmail.com}
         \author{Mark Macerato}
        \address{School of Mathematics, University of Minnesota, Twin cities, Minneapolis, MN 55455 }
         \email{macer008@umn.edu}
 \author{David Nadler}
        \address{Department of Mathematics, UC Berkeley, Evans Hall, Berkeley, CA 9472}
         \email{nadler@math.berkeley.edu}

\begin{document}

\begin{abstract}
We prove relative Langlands duality for
$X = GL_{p+q}/(GL_p \times GL_q)$.
These symmetric spaces arise in the setting of the linear period of Friedberg–Jacquet when $p=q$, and exhibit non-tempered behavior when $q>p+1$. Via the real-symmetric correspondence, we establish a derived Satake equivalence for the real group $G_\bR=U(p,q)$, and use it to obtain a constructible realization of coherent sheaves on Kato's exotic nilpotent cones. Along the way, we describe the dual variety of $X$ in terms of the Whittaker cotangent bundle and regular centralizers, viewed as Coulomb branches.

\end{abstract}
\maketitle
\tableofcontents
\setcounter{tocdepth}{1}


\section{Introduction}


\subsection{Relative Langlands duality} 
The relative Langlands program of Ben-Zvi, Sakellaridis and Venkatesh~\cite{BZSV}  
proposes geometric mechanisms for observed and conjectural relations 
between periods and L-functions.

A fundamental conjecture in the program predicts  that given a 
complex reductive group $G$, with dual group $\check G$, and a 
homogeneous
spherical $G$-variety $X$, there exists 
a (graded) Hamiltonian $\check G$-variety
$\check M$  with moment map 
$\mu:\check M\to\check\fg^*$ equipped with 
a commuting $\mathbb G_m$-action satisfying $\mu(t\cdot m)=t^2\mu(m)$,
and a relative derived Satake equivalence 
\begin{equation}\label{eq:rel sat}
D^b_{\frak L^+G}(\frak LX)\cong \mathrm{Perf}^{\check G}(\check M)^\renfatslash
    \end{equation}

 On the left or automorphic side, $\frak LX$ is the formal loop space of $X$,  $\frak L^+G$ is the formal arc space of $G$, and $D^b_{\frak L^+G}(\frak LX)$ is the bounded constructible dg-derived category of $\frak L^+G$-equivariant  sheaves on $\frak LX$. On the right  or spectral/Galois side,  $\mathrm{Perf}^{\check G}(\check M)^\renfatslash$ is the sheared dg-category of
$\check G$-equivariant perfect dg-modules over the 
ring of functions on $\check M$ viewed as a dg-algebra
with trivial differential and grading given by the $\mathbb G_m$-action.
The shearing notation and the shifted vector spaces used below are
defined in Section~\ref{sect:shearing}; the parity convention is
specified in Section~\ref{sect:abelian}. 

Moreover, the relative derived Satake equivalence \eqref{eq:rel sat}  should be 
compatible with the natural monoidal actions under the  derived Satake equivalence 
$$D^b_{\frak L^+G}(\Gr_G)\cong 
\on{Perf}^{\check G}(\check\fg^*[2])$$ 
where $\Gr_G = \frak L G/\frak L^+G $ is the affine Grassmannian,  $D^b_{\frak L^+G}(\Gr_G)$ acts 
on  $D^b_{\frak L^+G}(\frak LX)$ by
convolution, and  $\on{Perf}^{\check G}(\check\fg^*[2])$ acts on  $\on{Perf}^{\check G}(\check M)^\renfatslash$ by 
tensor product  through $\mu$.

\subsection{Linear period of Friedberg--Jacquet} 
An important  instance of relative duality is the linear period of  Friedberg--Jacquet~\cite{FJ}:
here one takes  $G=GL_{2p}$, so with dual group $\check G=GL_{2p}$,  and  studies $X = GL_{2p}/(GL_p \times GL_p)$ with  expected dual 
$$
\check M = GL_{2p} \times^{\Sp_{2p}} ( \wedge^2\mathbb A^{2p}[2]\times T^*\mathbb A^{2p}[1] )
$$
$$
\check M/ \check G =  ( \wedge^2\mathbb A^{2p}[2]\times T^*\mathbb A^{2p}[1]  )/ \Sp_{2p}
$$
Similar phenomena can be observed  more generally, when 
one  takes $G=GL_{p+q}$, $q\geq p$, so with dual group  $\check G=GL_{p+q}$,  and studies $X =GL_{p+q}/(GL_p \times GL_q)$  with  expected dual
$$
\check M = GL_{p+q} \times^{\Sp_{2p}} ( \wedge^2\mathbb A^{2p}[2]\times T^*\mathbb A^{2p}[q-p + 1] \times\frak{gl}^*_{q-p}[2]\sslash GL_{q-p}  )
$$
$$
\check M/ \check G =  ( \wedge^2\mathbb A^{2p}[2]\times T^*\mathbb A^{2p}[q-p + 1] \times\frak{gl}^*_{q-p}[2]\sslash GL_{q-p}  )/ \Sp_{2p}
$$
In the associated-bundle presentation the grading also acts on the
group factor by $g\mapsto g\lambda(t)^{-1}$, with $\lambda$ as in
\eqref{rho}; after quotienting by $\check G$, only the displayed fiber
degrees remain. When $p=q$, the final $GL_{q-p}$ quotient factor
is a point.
Note this is tempered if and only if $q=p$ or $q=p+1$.

Our main result is a relative derived Satake equivalence in this setting:

\begin{theorem}(see Theorem~\ref{main}) \label{intro main symmetric}
For $G=GL_{p+q}$, $q\geq p$, and $X = GL_{p+q}/(GL_p \times GL_q)$, 
there is a relative derived Satake equivalence of dg-categories
\[D^b_{\frak L^+ G}(\frak L X) 
\cong\mathrm{Perf}^{\Sp_{2p}}(\wedge^2\mathbb A^{2p}[2]\times T^*\mathbb A^{2p}[q-p+1]\times\frak{gl}^*_{q-p}[2]\sslash GL_{q-p})\]
compatible with the natural  monoidal actions under  the 
derived Satake equivalence
$$
D^b_{\frak L^+G}(\Gr_G)\cong 
\on{Perf}^{\check G}(\check\fg^*[2])
$$
\end{theorem}

Within the above derived Satake equivalence, we also find an abelian equivalence.
According to \cite[Section 8.4]{CN}, the category  $D_{\frak L^+G}^b(\frak LX)$ admits a
perverse $t$-structure.  We denote the corresponding category of perverse sheaves by $\on{Perv}_{\frak L^+G}(\frak LX)$.
\begin{theorem}(see Theorem \ref{abelian equ})
There is an equivalence of abelian category
 \[
	\mathrm{Perv}_{\frak L^+G}(\frak LX)\cong
	\begin{cases}
		\wedge^*(T^*\mathbb A^{2p})\on{-mod}(\mathrm{Rep}(\Sp_{2p}))\ \ \ \ \ p=q,\\
		\mathrm{Rep}(\Sp_{2p})\ \ \ \ \ \ \ \ \ \ \ \ \ \ \ \ \ \ \ \ \ \ \ \ \ \ \ \ p\neq q.
	\end{cases}
\]
\end{theorem}


\subsection{Derived Satake equivalence for $U(p,q)$}

A  key feature of  the linear period of  Friedberg--Jacquet,
and more generally  of $G= GL_{p+q}$ acting on $X = GL_{p+q}/(GL_p \times GL_q)$ is that $X$ is a  symmetric space, i.e.~$K=GL_p \times GL_q$ is the fixed-points of an involution $\theta$ of $G$.
Among many consequences, this allows one  to reinterpret the relative derived Satake equivalence of Theorem~\ref{intro main symmetric} in terms of the unitary group 
$G_\bR=U(p,q)$, which is quasi-split if and only if $q=p$ or $q=p+1$.
Following~\cite{CN} and recalled as Proposition~\ref{real-symmetric} below, the automorphic side admits a natural equivalence
\begin{equation}\label{eq: intro mat}
D^b_{\frak L^+ G}(\frak L X) 
\cong D^b_{\frak L^+U(p,q)}( \Gr_{U(p,q)})
\end{equation}
where $\Gr_{U(p,q)}  =\frak L U(p,q)/ \frak L^+U(p,q)$ is the real affine Grassmannian of $U(p,q)$.
Moreover, the equivalence \eqref{eq: intro mat} is compatible with natural actions of the derived Satake category.

So the relative derived Satake equivalence of Theorem~\ref{intro main symmetric} can be restated as  a real derived Satake equivalence:

\begin{theorem}[see Theorem~\ref{main}]\label{intro main real}
For $G_\bR = U(p, q)$ 
there is a real derived Satake equivalence of dg-categories
$$
 D^b_{\frak L^+U(p,q) }(\Gr_{U(p,q)})
\cong\mathrm{Perf}^{\Sp_{2p}}(\wedge^2\mathbb A^{2p}[2]\times T^*\mathbb A^{2p}[q-p+1]\times\frak{gl}^*_{q-p}[2]\sslash GL_{q-p})$$
compatible with the natural  monoidal actions under  the 
derived Satake equivalence
$$
D^b_{\frak L^+G}(\Gr_G)\cong 
\on{Perf}^{\check G}(\check\fg^*[2])
$$

\quash{
(2)    There is an equivalence of dg categories
\[D_{\frak L^+U(p,q)}(\Gr_{U(p,q)})\cong\mathrm{QCoh}^{\check G}(\check M^\renfatslash)_{\mathcal N_{\check\fg^*}}\cong\]
\[\cong\mathrm{QCoh}^{\Sp_{2p}}(\wedge^2\mathbb A^{2p}[2]\times T^*\mathbb A^{2p}[q-p+1]\times\frak{gl}^*_{q-p}[2]\sslash GL_{q-p})_{\mathcal N_{\Sp_{2p}}'}\]
compatible with the action of 
derived Satake equivalence 
$D_{\frak L^+G}(\Gr_{G}))\cong\mathrm{QCoh}^{\check G}(\check\fg^*[2])_{\mathcal N_{\check\fg^*}}$ on both sides.
Moreover, it restricts to an equivalence 
\[D_{\frak L^+U(p,q)}(\Gr_{U(p,q)})^c\cong\mathrm{Perf}^{\check G}(\check M^\renfatslash)_{\mathcal N_{\check\fg^*}}\cong\]
\[\cong
\mathrm{Perf}^{\Sp_{2p}}(\wedge^2\mathbb A^{2p}[2]\times T^*\mathbb A^{2p}[q-p+1]\times\frak{gl}^*_{q-p}[2]\sslash GL_{q-p})_{\mathcal N_{sp_{2p}}'}\]
on compact objects.
}
\end{theorem}

According to \cite{N}, there is 
natural perverse $t$-structure on $D^b_{\frak L^+U(p,q)}(\Gr_{U(p,q)})$
and we denote by
$\mathrm{Perv}_{\frak L^+U(p,q)}(\Gr_{U(p,q)})$
the corresponding category of perverse sheaves.

\begin{theorem}(see Theorem \ref{abelian equ})
There is an equivalence of abelian category
 \[
	\mathrm{Perv}_{\frak L^+U(p,q)}(\Gr_{U(p,q)})\cong
	\begin{cases}
		\wedge^*(T^*\mathbb A^{2p})\on{-mod}(\mathrm{Rep}(\Sp_{2p}))\ \ \ \ \ p=q,\\
		\mathrm{Rep}(\Sp_{2p})\ \ \ \ \ \ \ \ \ \ \ \ \ \ \ \ \ \ \ \ \ \ \ \ \ \ \ \ p\neq q.
	\end{cases}
\]
\end{theorem}
\subsection{Whittaker cotangent bundles
and regular centralizers}

Another key feature of  the linear period of  Friedberg--Jacquet,
and more generally  of $G= GL_{p+q}$ acting on $X = GL_{p+q}/(GL_p \times GL_q)$ is that $K=GL_p \times GL_q$ is  a Levi subgroup of $G$.
The work of the third author \cite{M} 
suggests the alternative  description 
of the dual variety $\check M$ as 
 the affine closure of the Whittaker cotangent bundle 
 \begin{equation}\label{intro: M=Wh}
    \check M\cong  \overline{T^*_{\psi_{\check K}}(\check G/\check U)}
 \end{equation}
 here 
 $T^*_{\psi_{\check K}}(\check G/\check U)$ is  the Hamiltonian reduction of $T^*\check G$ by a maximal unipotent subgroup $\check U =[\check B, \check B] \subset \check G$ at $-d\psi_{\check K}$ for the right action $g\mapsto gu^{-1}$,
where $\psi_{\check K}:\check U\to\mathbb G_a$ is an additive character nontrivial precisely on the simple roots of the dual Levi $\check K =GL_p \times GL_q$
 (see Section~\ref{sect:whit cot} for more details). 

On the other hand, a recent general result of Ginzburg \cite{Gi}
suggests another description of 
the dual variety in terms of regular centerlizers group schemes:
\begin{equation}\label{intro: M=J}
  \check M\cong \overline{\check G\times J_{\check K}/\pi^*J_{\check G}}
\end{equation}
here the right hand side is  the affine closure of the quotient $\check G\times J_{\check K}/\pi^*J_{\check G}$, where $J_{\check G} \to \mathfrak c_{\check G}$, $J_{\check K} \to \mathfrak c_{\check K}$ are the group schemes of regular centralizers, and $\pi:\mathfrak c_{\check K} \to \mathfrak c_{\check G}$ is the natural projection (see Section~\ref{sect: reg cent} for more details).

Indeed, we prove the isomorphisms~\eqref{intro: M=Wh} and ~\eqref{intro: M=J}, which is the key step in our approach to Relative Langlands duality for linear periods:

\begin{theorem}[see Theorem~\ref{affine closures}]\label{intro affine closures}
    There are $\check G\times\mathbb G_m$-equivariant 
isomorphisms of 
affine varieties over $\check\fg^*$
in the following commutative diagram 
\[\xymatrix{\overline{\check G\times J_{\check K}/\pi^*J_{\check G}}\ar[rr]^{\simeq}\ar[dr]_{\simeq}&&\overline{T^*_{\psi_{\check K}}(\check G/\check U)}\ar[ld]^{^{\simeq}}\\
&\check M&}\]
\end{theorem}

In fact,
the isomorphism 
$\overline{T^*_{\psi_{\check K}}(\check{G}/\check{U})}\cong\overline{(\check G\times J_{\check K})/\pi^*J_{\check G}}$ in
Theorem \ref{intro affine closures}
follows from the work in 
 \cite{Ga,M}, which is valid when $\check K$
 is a general Levi subgroups (not necessarily symmetric or spherical), and  the main step 
 in the proof is the isomorphism~\eqref{intro: M=Wh}.
 The proof uses a generalization of Kostant theorem, see Section \ref{generalized Kostant thm}, and a Lagrangian correspondence between $T^*_{\psi_{\check K}}(\check G/ \check U)$ and $\check M$, see  Section \ref{Lag}.

Since $\check M$ is smooth,
as a corollary of the theorem we obtain the following:
\begin{corollary}[see Corollary~\ref{cor: aff clos}]
    The affine closures $\overline{T^*_{\psi_{\check K}}(\check{G}/\check{U})}\cong\overline{(\check G\times J_{\check K})/\pi^*J_{\check G}}$ are smooth.
\end{corollary}

\begin{remark}
It is proved in \cite[Corollary 3.1.6]{Gi}
that the affine closure 
$\overline{\check G\times J_{\check K}/\pi^*J_{\check G}}$ for general connected reductive subgroup $\check K\subset\check G$
is an irreducible normal variety.
However, the smoothness of the affine closures in setting of Levi symmetric subgroup $\check K=\GL_p\times\GL_q$ was unexpected at least by the authors.
\end{remark}


\subsection{Exotic nilpotent cones}
Finally, this paper also makes contact with the exotic nilpotent cones of~\cite{Ka}.
Recall the dual variety 
$$
\check M/ \check G =  ( \wedge^2\mathbb A^{2p}[2]\times T^*\mathbb A^{2p}[q-p + 1] \times\frak{gl}^*_{q-p}[2]\sslash GL_{q-p}  )/ \Sp_{2p}
$$
The null-cone for the $\Sp_{2p}$ action on 
$ \wedge^2\mathbb A^{2p}[2]\times T^*\mathbb A^{2p}[q-p + 1] $, denoted by \[\mathcal N_{\Sp_{2p}}'\subset \wedge^2\mathbb A^{2p}[2]\times T^*\mathbb A^{2p}[q-p + 1]\] is known as the 
exotic nilpotent cones of~\cite{Ka}. 
Consider the dg category $D^b_{}( \Gr_{U(p,q)})$ of bounded constructible complexes on $\Gr_{U(p,q)}$
and its
full subcategory 
$$D^b_{\{\frak L^+U(p,q)\}}( \Gr_{U(p,q)})\subset D^b_{}( \Gr_{U(p,q)})$$
of bounded  complexes constructible with respect to the 
$\frak L^+U(p,q)$-orbits stratification.
As a corollary of derived Satake for $U(p,q)$,
we obtain the following constructible description coherent sheaves on $\mathcal N_{\Sp_{2p}}'$:
\begin{theorem}(see Theorem \ref{exotic-coherent})
\label{thm: intro main with supp}
  There is an equivalence of dg-categories
$$
D^b_{\{\frak L^+U(p,q)\}}( \Gr_{U(p,q)})
\cong\mathrm{Perf}^{\Sp_{2p}}(\mathcal N_{\Sp_{2p}}')$$
\end{theorem}

\subsection{Hermitian symmetric spaces}
The symmetric space $X=GL_{p+q}/(GL_p\times GL_q)$ considered in this paper is a particular example of a Hermitian symmetric space. In forthcoming work \cite{CDMN}, we expect to establish Relative Langlands duality for Hermitian symmetric spaces with no roots of type $N$.
Modulo central factors, the Hermitian symmetric spaces under consideration are as follows:
\begin{table}[h]
\begin{tabular}{l|l}
$X=G/K$ & $\text{Rank}$     \\ \hline
$\GL_{2r+m}/(\GL_r \times \GL_{r+m})$ & $r$     \\ \hline
$\on{PSO}_{2n}/\GSO_{2n-2}$ & $2$    \\ \hline
$\SO_{4r}/\GL_{2r}$ & $r$    \\ \hline
$\SO_{4r+2}/\GL_{2r+1}$ & $r$     \\ \hline
$E_6/\GSO_{10}$ & $2$   \\ \hline
$E_7/GE_6$ & $3$   
\end{tabular}

\label{tab: hermitian weights}
\end{table}

A striking feature of the Hermitian symmetric case is that, after forgetting the gradings, all of the corresponding dual varieties $\check M$ are of the form
$$\ld{G} \times^{\Sp(V)} (T^\ast(V) \times \wedge^2(V) \times \mathbf A^{a_X})$$
where $V$ is a vector space of dimension $2r$ and $a_X$ is the
dimension of the remaining affine factor. For the linear periods
considered here, $r=p$ and $a_X=q-p$.

\subsection{Acknowledgements}
T.-H. Chen would like to thank the NCTS-National Center for Theoretical Sciences at Taipei where parts of this work were done.
THC is supported by NSF grant DMS-214372
and Simons Fellowships.
SKD was supported by an NSF postdoc fellowship and a Simons postdoc fellowship (as part of the Simons Collaboration on Perfection in Algebra, Geometry, and Topology).
MM was supported by an NSF postdoc fellowship.
DN is grateful for the support of NSF  grant DMS-2401178.


\section{Constructible side}

Throughout, sheaves have complex coefficients, and $\mathrm{Rep}(H)$
denotes finite-dimensional algebraic representations of a complex
algebraic group $H$. We use $D^b$ for the bounded, finite-rank
constructible category and $\mathrm{Perv}$ for its perverse heart.
The cocomplete categories $D$ and their compact objects are distinguished
below in Section~\ref{sect:satake-categories}.

\subsection{Real groups and symmetric spaces} 
Let $G=GL_N$ be the complex general linear group of rank $N$.
Let $T$ be the diagonal maximal torus and 
$B=TU$ be the Borel subgroup of upper triangular matrices.

Fix integers $1 \leq p \leq q$ and put $N=p+q$.
Consider the following matrix
\[
I_{p,q}  = \begin{pmatrix}
    0 & 0 & J_p\\
    0 & I_{q-p} & 0\\
    J_p & 0 & 0
\end{pmatrix}  
\]
where $I_{q-p}$ is the identity matrix 
of size $q-p$
and $J_p$ is the permutation matrix with $1$'s on the anti diagonal and $0$'s elsewhere. Consider the involution 
$\theta: G\to G$ given by $\theta(g)=I_{p,q} g (I_{p,q})^{-1}$.
The fixed point subgroup 
$K=G^{\theta}$ is a Levi symmetric subgroup of $G$ isomorphic to 
$K\cong GL_p\times GL_q$.
We denote by $X=G/K$ the associated symmetric space.

The block upper triangular subgroup 
\[
P  = \begin{pmatrix}
    B_p^{\mathrm{}} & * & *\\
    0 & GL_{q-p} & *\\
    0 & 0 & B_p^{\mathrm{}}
\end{pmatrix}  
\]
is a $\theta$-split minimal parabolic subgroup, where
$B_p^{\mathrm{}}\subset GL_p$ is the upper triangular Borel. Its Levi subgroup is
\[L=P\cap\theta(P)=\begin{pmatrix}
    GL_1^p & 0 & 0\\
    0 & GL_{q-p} & 0\\
    0 & 0 & GL_1^p
\end{pmatrix} .\]
The subtorus $A=\{t\in T|\theta(t)=t^{-1}\}^\circ\subset T$ is a maximal $\theta$-split torus consisting of  diagonal matrices of the form
\[\text{diag}(a_1,...,a_p,1,...,1,a_p^{-1},...,a_1^{-1})\]
We have
$L=Z_G(A)$
and $P=LU_P=MAU_P$
where $U_P$ is the unipotent radical of $P$
and $M=Z_K(A)$.
We see that $X$ is quasi-split (but non-split)
if and only if $p=q$ or $p+1=q$.

We write $\Lambda_T=X_*(T)$ for the coweight lattice and
$\Lambda_T^+$ for the dominant coweights with respect to $(T,B)$.
Let $2\rho\in X^*(T)$ be the sum of the positive roots. The real
coweight lattice is $\Lambda_A=X_*(A)$, with
$\Lambda_A^+=\Lambda_A\cap\Lambda_T^+$.
Under the coordinate identifications $X_*(T)\cong\bZ^N$ and
$X^*(T)\cong\bZ^N$, we have
\[
 2\rho=(N-1,N-3,\ldots,-N+3,-N+1).
\]
The elements of $\Lambda_A$ are the tuples
\begin{equation}\label{Real Coweights}
 (\lambda_1,\ldots,\lambda_p,0,\ldots,0,-\lambda_p,\ldots,-\lambda_1).
\end{equation}

We have an embedding 
$\tau:X\to G$ sending $x=gK$
to $g\theta(g)^{-1}$ which induces an isomorphism 
$\tau:X\cong Q=\{g\in G|\theta(g)=g^{-1}\}^\circ$ 
between $X$ and the 
neutral component of the $\theta$ anti-fixed points in $G$. 
We denote by  $\iota_A:A\to Q\cong X$ the natural embedding.

\quash{
Let $V = \bC^{N}$ and let $G = GL(V) = GL_{N}(\bC)$. Fix integers $1 \leq p \leq q$ and put $N=p+q$. Let $W_1 = \bC^p \oplus 0 \subseteq V$ and $W_2 = 0 \oplus \bC^q \subseteq V$. We will consider the standard Levi subgroup
$$
K = GL(W_1) \times GL(W_2) = GL_p(\bC) \times GL_q(\bC) \subseteq G.
$$
It is a symmetric subgroup of $G$. Indeed, letting $\xi \in G$ act on $W_1$ (resp. $W_2$) by $+1$ (resp. $-1$), we can exhibit $K$ as the subgroup of fixed points of the involution $\theta = \mathrm{Ad}_{\xi}$. Let $X = G/K$ be the associated (complex) symmetric space. 
}

Consider the conjugation 
$\eta_c:G\to G$ given by 
$\eta_c(g)=(\bar g^t)^{-1}$
with fixed points subgroup 
$G_c=G^{\eta_c}=U(N)$
the compact unitary group.
Consider the conjugation 
$\eta:G\to G$ given by 
$\eta(g)=I_{p,q} (\bar g^{-1})^tI_{p,q}^{-1}$.
The fixed point subgroup $G^{\eta}$ 
is a real form of $G$, denote by $G_\bR=U(p,q)$, known as the 
 unitary group
of signature $(p, q)$.
Consider the intersections
$P_\bR=P\cap G_\bR$, $L_\bR=L\cap G_\bR$, and
$U_{P,\bR}=U_P\cap G_\bR$, etc.
The intersection $P_\bR=P\cap G_\bR$ is the minimal parabolic subgroup with 
Iwasawa decomposition 
$P_\bR=M_\bR A_\bR U_{P,\bR}$. 
The intersection 
$K_c=K\cap G_\bR$ is a maximal compact subgroup isomorphic to the product 
$K_c\cong U(p)\times U(q)$ of unitary groups.

\subsection{Loop spaces and affine Grassmannians} Let $\mathfrak{L}G$ and $\mathfrak{L}K$ denote the loop groups of $G$ and $K$, respectively. We also have the arc subgroups $\mathfrak{L}^+G \subseteq \mathfrak{L}G$ and $\mathfrak{L}^+K \subseteq \mathfrak{L}K$. Let $\Gr_G = \mathfrak{L}G/\mathfrak{L}^+G$ and $\Gr_K = \mathfrak{L}K/\mathfrak{L}^+K$ be the corresponding affine Grassmannians. We abbreviate $\Gr=\Gr_G$. There is a natural closed immersion
$\iota_K:\Gr_K\hookrightarrow\Gr_G$. 
Let $\mathfrak LG_\bR$ be the real loop group and
$\mathfrak L^+G_\bR\subseteq\mathfrak LG_\bR$ its arc subgroup. The quotient $\Gr_\bR = \mathfrak{L}G_\bR/\mathfrak{L}^+G_{\bR}$ is the real affine Grassmannian, a real form of the ind-scheme $\Gr_G$. 
For any $\lambda\in\Lambda_T^+$
(resp. $\lambda\in\Lambda_A^+$)
we denote by $\Gr^\lambda\subset\Gr$ (resp. $\Gr_\bR^\lambda\subset\Gr_\bR$) the $\frak L^+G$-orbit (resp. $\frak L^+G_\bR$) through $t^\lambda$. 
We have $\Gr=\bigsqcup_{\lambda\in\Lambda^+_T}\Gr^\lambda$
and $\Gr_\bR=\bigsqcup_{\lambda\in\Lambda^+_A}\Gr_\bR^\lambda$.

Let $\frak LX$ be the loop space of $X$. 
The embedding $\iota_A:A\to X$
of the $\theta$-split torus 
gives rise to an embedding 
$\Lambda_A\to\frak LX$
sending $\lambda$ to $x^\lambda=\iota_A(t^\lambda)$. Let $\frak L X^\lambda=\frak L^+G\cdot 
x^\lambda$ be the $\frak L^+G$-orbit through $x^\lambda$.
It is known that
$\frak LX=\bigsqcup_{\lambda\in\Lambda_A^+} \frak LX^\lambda$ 
and $\overline{\frak L X^\lambda}=
\bigsqcup_{\mu\in\Lambda_A^+,\mu\leq\lambda} \frak LX^\mu$
(see, e.g., \cite[Proposition 17 and Theorem 18]{CY}).

\subsection{Real and relative Satake categories}\label{sect:satake-categories}
Let $(D_{\frak L^+G}(\Gr),*)$
be the co-complete monoidal
dg category of 
$\frak L^+G$-equivariant 
 constructible complexes on 
$\Gr$
with monoidal structure given by the 
convolution product:
$F,F'\in D_{\frak L^+G}(\Gr)$ and 
$F\star F'=m_!(F\tilde\boxtimes F')$, where 
$m:\frak LG\times^{\frak L^+G}\Gr\to\Gr$ is the natural action map and 
$F\tilde\boxtimes F'\in D_{\frak L^+G}(\frak LG\times^{\frak L^+G}\Gr)$
is the twisted product of $F$ and $F'$.
It is known that $D_{\frak L^+G}(\Gr)$ is compactly generated and we denote by $D_{\frak L^+G}(\Gr)^c$
the subcategory of compact objects.
The equivariant derived category
$D^b_{\frak L^+G}(\Gr)$ coincides with the 
subcategories of 
locally compact objects\footnote{An object 
$F$ is called locally compact if its image under the forgetful functor
$\mathrm{For}:D_{\frak L^+G}(\Gr)\to D(\Gr)$ is compact.} and we denote by $\text{Ind}(D^b_{\frak L^+G}(\Gr))$ the ind-completion of $D^b_{\frak L^+G}(\Gr)$.
We denote by $(D_{\frak L^+G_\bR}(\Gr_\bR),*), $ etc, the corresponding monoidal categories for $\Gr_\bR$.
We will call the categories $D_{\frak L^+G}(\Gr)$, etc
(resp. $D_{\frak L^+G_\bR}(\Gr_\bR)$, etc)
the complex Satake categories for $G$ (resp. real Satake categories for $G_\bR$).

Let $D_{\frak L^+G}(\frak LX)$
be the co-complete 
dg category of 
$\frak L^+G$-equivariant 
 constructible complexes on 
$\frak LX$. Denote by
$D_{\frak L^+G}(\frak LX)^c$, $D_{\frak L^+G}^b(\frak LX)$, 
and $\mathrm{Ind}(D^b_{\frak L^+G}(\frak LX))$ for the corresponding
categories of compact objects, locally compact objects, and the
ind-completion of the latter, respectively. 
We have a natural action of the complex Satake category 
\[*:D_{\frak L^+G}(\Gr)\times 
D_{\frak L^+G}(\frak LX)\to  D_{\frak L^+G}(\frak LX)\]
given by 
$F\star F'=a_!(F\tildeboxtimes F')$ where 
$a:\frak LG\times^{\frak L^+G}\frak LX\to\frak LX$ is the action map and 
$F\tildeboxtimes F'\in D_{\frak L^+G}(\frak LG\times^{\frak L^+G}\frak LX)$
is the twisted product of $F$
and $F'$.
We will call the above category 
$D_{\frak L^+G}(\frak LX)$, etc, the relative Satake categories.

\subsection{Hecke actions}
Let $\Gr^{(2)}\to\bC$ 
be the  Beilinson-Drinfeld Grassmannian
over the complex line 
classifying triples $(x,E,\phi)$
where $x\in\bC$, $E$ is a $G$-bundle on 
$\bC$, and $\phi$ is a trivialization of 
$E$ on $\bC\setminus\{\pm x\}$.
Consider the anti-holomorphic involution $\eta$ on $\Gr^{(2)}$ given by $\eta(x,E,\phi)=(-\bar x,\eta^*E,\eta^*\phi)$ defines a real from 
$\Gr^{(2)}_\bR\to i\bR$ such that
\[\xymatrix{\Gr\times i\bR_{>0}\ar[r]^a\ar[d]&\Gr_\bR^{(2)}\ar[d]&\Gr_\bR\ar[d]\ar[l]_b\\
i\bR_{>0}\ar[r]&i\bR&\{0\}\ar[l]}\]
Consider the specialization functor 
\begin{equation}\label{sp}
    \psi:D_{\frak L^+G}(\Gr)\to D_{\frak L^+G_\bR}(\Gr_\bR)\ \ \ \psi(F)=b^*a_*(F\boxtimes \bC_{i\bR_{>0}}).
\end{equation}

Let $\text{Perv}_{\frak L^+G}(\Gr)\subset D^b_{\frak L^+G}(\Gr)$ be the category of $\frak L^+G$-equivariant perverse sheaves on $\Gr$. The geometric Satake equivalence says that there is an isomorphism of Tannakian category 
\begin{equation}\label{GS}
    \text{Perv}_{\frak L^+G}(\Gr)\cong\text{Rep}(\check G)
\end{equation}
where $\check G=GL_N$ is the dual group of $G$ such that the intersection cohomology complex $\text{IC}^\lambda$ for $\Gr^\lambda$
maps to the irreducible representations 
$V_\lambda$ of highest weight $\lambda$. For
$V\in\mathrm{Rep}(\check G)$, write $\IC_V$ for its corresponding
Satake sheaf.
According to \cite[Lemma 3.9.1]{N},
we have 
\[\dim_\bR(\Gr_\bR^\lambda)=\langle\lambda,2\rho\rangle=2\sum_{i=1}^p(N+1-2i)\lambda_i\in 2\mathbb Z.\] 
Thus there is perverse $t$-structure 
on the real Satake category 
$D^b_{\frak L^+ G_\bR}(\Gr_\bR)$
given by the middle perversity function
\[p_\bR:\Lambda_A^+\to\mathbb Z,\ \ \ 
p_\bR(\lambda)=-\frac{\dim_\bR(\Gr_\bR^\lambda)}{2}=-\sum_{i=1}^p(N+1-2i)\lambda_i\]
We denote by 
$\text{Perv}_{\frak L^+G_\bR}(\Gr_\bR)$
the corresponding category of $\frak L^+G_\bR$-equivariant perverse sheaves on $\Gr_\bR$.

The Hecke action of $\mathrm{Rep}(\check G)$
on the real Satake category
$D_{\frak L^+G_\bR}(\Gr_\bR)$, denoted again by $*$, is defined as 
\[\mathrm{Rep}(\check G)\times D_{\frak L^+G_\bR}(\Gr_\bR)\to D_{\frak L^+G_\bR}(\Gr_\bR)\ \ (V,F)\to V*F:=\psi(\IC_V)*F\]

\begin{proposition}\label{exactness}
 (1)  The specialization functor $\psi$, and hence the Hecke action, 
is monoidal. Moreover, the Hecke action 
is $t$-exact with respect to the perverse $t$-structure if and only if $G_\bR$ is quasi-split, that is,
when $G_\bR=U(p,p)$ or $U(p,p+1)$.

(2) The Hecke action of $\mathrm{Rep}(\check G)$
preserves semisimplicity, that is, it maps semi-simple objects to semi-simple objects.
\end{proposition}

\begin{proof}
Part (1) is \cite[Theorem 1.2.2 and Corollary 10.6.2]{N}
and Part (2)
is \cite[Corollary 1.5]{CN}.
\end{proof}

We shall show that the real Satake category 
$D_{\mathfrak{L}^+G_\bR}(\Gr_\bR)$ is generated by 
 $\IC^0_\bR$ under the Hecke action.

\begin{lemma}\label{real coweights}
     Any real dominant coweight $\lambda\in\Lambda_A^+$ is of the form 
    $\lambda=\mu-\theta(\mu)$ for some $\mu\in\Lambda_T$
\end{lemma}
\begin{proof}
The involution $\theta$ 
sends a coweight $\mu=(a_1,...,a_{p},b_{1},...,b_{q-p},c_{1},...c_p)\in\Lambda_T\cong \mathbb Z^N$ to
\[\theta(\mu)=(c_p,...,c_1,b_1,...,b_{q-p},a_p,...,a_1
).\]
It follows that 
for any real dominant coweight $\lambda=(\lambda_1,...\lambda_p,0,...,0,-\lambda_p,...,-\lambda_1)\in\Lambda_A^+$,  we have 
$\lambda=\mu-\theta(\mu)$ where 
$\mu=(\lambda_1,...,\lambda_p,0,...,0)\in\Lambda_T$.
The lemma follows.
\end{proof}

\begin{lemma}\label{connectedness}
The $\mathfrak L^+G_\bR$-stabilizers on $\Gr_\bR$ is connected.
\end{lemma}
\begin{proof}
According to \cite[Proposition 3.6.1]{N}, each
$\mathfrak L^+G_\bR$-orbit $\Gr_\bR^\lambda$ is a
$G_\bR$-equivariant vector bundle over a real flag manifold
$G_\bR/P_{\bR,\lambda}$, where $P_{\bR,\lambda}$ is a real parabolic
subgroup depending on $\lambda$. Thus it suffices to check that
$P_{\bR,\lambda}$ is connected. Since a real parabolic is connected
if and only if its Levi subgroup is connected, we reduce to the latter
assertion. The Levi subgroups of $G_\bR$ are isomorphic to
\[
 U(p-a,q-a)\times\prod_i GL_{n_i}(\mathbb C),
 \qquad a=\sum_i n_i\leq p,
\]
which is connected.
\end{proof}

\begin{proposition}\label{generation}
(1)
The
irreducible objects of 
$\mathrm{Perv}_{\frak L^+G_\bR}(\Gr_\bR)$
are given by the the intersection cohomology complex $\IC_\bR^\lambda$ for 
$\Gr_\bR^\lambda$, $\lambda\in\Lambda_A^+$.

(2)
 The real Satake category 
$D
_{\mathfrak{L}^+G_\bR}(\Gr_\bR)$ is generated by 
 $\IC_\bR^0$ under the Hecke action of $\mathrm{Rep}(\check G)$
\end{proposition}
    \begin{proof}
    Lemma \ref{connectedness} implies that  $D_{\mathfrak{L}^+G_\bR}(\Gr_\bR)$ is generated by $\mathrm{IC}$-complexes $\IC_\bR^\lambda$ and Lemma \ref{real coweights} together with \cite[Corollary 10.6.2]{N} implies, up to shifts, all the $\mathrm{IC}$-complexes appear as a direct summand of $V\star \IC_\bR^0$ for some $V\in\mathrm{Rep}(\check G)$. The corollary follows.
\end{proof}

It is shown in \cite[Theorem 3.6]{CY} that 
the collection of spherical orbits $\{\frak LX^\lambda\}_{\lambda\in\Lambda_A^+}$
provides a 
placid stratification of $\frak LX$ (see \cite[Appendix A]{CY}
for a review  of placidness).
Thus by \cite[Proposition 8.12.]{CN}, there is a unique perverse $t$-structure on 
on $D_{\frak L^+G}(\frak LX)$ associated to the 
perversity function
\[p_X:\Lambda_A^+\to\mathbb Z,\ \ p_X(\lambda)=-\langle\lambda,\rho\rangle\]
characterized by the property that 
\[^{p}D^{\geq0}_{\frak L^+ G}(\frak LX)=\{F\in D^{}_{\frak L^+ G}(\frak LX)| i^!_\lambda(F)\in {^{p}D}^{\geq0}_{\frak L^+ G}(\frak LX^\lambda)[p_X(\lambda)]\}\]
here $i_\lambda:\frak LX^\lambda\to\frak LX$ is the embedding and 
${^{p}D}^{\geq0}_{\frak L^+ G}(\frak LX^\lambda)$
is the positive part of the 
$!$-adapted perverse $t$-structure 
on the placid scheme $\frak LX^\lambda$.
We denote by 
$\mathrm{Perv}_{\frak L^+G}(\frak LX)$ the corresponding perverse heart
on $D^b_{\frak L^+G}(\frak LX)$.
For any $\lambda\in\Lambda_A^+$, we denote by 
$\IC_X^\lambda$ the $\IC$-complex for the orbit 
$\frak LX^\lambda$.
We have $i^!_\lambda(\IC_X^\lambda)\cong\omega_{\frak LX^\lambda}[p_X(\lambda)]$, where $\omega_{\frak LX^\lambda}$ is the dualizing complex of $\frak LX^\lambda$.

Recall the following real-symmetric equivalence \cite[Theorem 1.1]{CN}:
\begin{proposition}\label{real-symmetric}
    There is a natural equivalence 
    \[D_{\frak L^+G}(\frak LX)\cong D_{\frak L^+G_\bR}(\Gr_\bR)\]
    which is exact with respect to the perverse $t$-structure and is compatible with the monoidal action of the 
    complex Satake category 
    $D_{\frak L^+G}(\Gr)$.
    In particular, the above equivalence gives rise to an equivalence between perverse sheaves
    \[\mathrm{Perv}_{\frak L^+G}(\frak LX)\cong\mathrm{Perv}_{\frak L^+G_\bR}(\Gr_\bR).\]
\end{proposition}

Consider the Hecke action of 
$\mathrm{Rep}(\check G)$
on $D
_{\mathfrak{L}^+G}(\frak LX)$ given by 
\[\mathrm{Rep}(\check G)\times D
_{\mathfrak{L}^+G}(\frak LX)\to D
_{\mathfrak{L}^+G}(\frak LX),\ \ \  V*F=\IC_V*F\]
The following follows directly from Proposition \ref{exactness}, Proposition \ref{generation} and  Proposition \ref{real-symmetric}:

\begin{proposition}
    
(1)
The
irreducible objects of 
$\mathrm{Perv}_{\frak L^+G}(\frak LX)$
are given by the intersection cohomology complex $\IC_X^\lambda$, $\lambda\in\Lambda_A^+$.

(2) The Hecke action on $D
_{\mathfrak{L}^+G}(\frak LX)$ is $t$-exact if and only if $X$ is quasi-split, that is, if $X\cong GL_{2p}/(GL_p\times GL_p)$
or $GL_{2p+1}/(GL_p\times GL_{p+1})$.

(3)
 The relative Satake category 
$D
_{\mathfrak{L}^+G}(\frak LX)$ is generated by 
 $\IC_X^0\cong\omega_{\frak L^+X}$ under the Hecke action of $\mathrm{Rep}(\check G)$.
\end{proposition}

\section{Spectral side}
\subsection{Hamiltonian dual varieties}
We recall the construction in \cite{BZSV} of the Hamiltonian dual variety
of $X=GL_{p+q}/(GL_p\times GL_q)$. Put $N=p+q$, $\delta=q-p$, and
$\check G=GL_N$. Throughout the spectral calculations, choose the lower
triangular Borel $\check B$, with diagonal torus $\check T$ and unipotent
radical $\check U$, with simple root vectors $E_{i+1,i}$.
For a block of size $m$, set
\[
 \tau_m(t)=\operatorname{diag}(t^{1-m},t^{3-m},\ldots,t^{m-1}),\qquad
 e_m=\sum_{i=1}^{m-1}i(m-i)E_{i+1,i},\qquad
 f_m=\sum_{i=1}^{m-1}E_{i,i+1}.
\]
Thus $(e_m,[e_m,f_m],f_m)$ is a principal $\mathfrak{sl}_2$-triple and
$\tau_m$ is the sum of the positive coroots for the lower triangular Borel.

Let $\check L=I_{2p}\times GL_\delta\subset\check G$ be the Levi subgroup, and write
\begin{equation}\label{rho}
 \lambda(t)=2\rho_{\check L}(t)
 =\operatorname{diag}(I_{2p},\tau_\delta(t)).
\end{equation}
For $\delta=0$, put $\lambda=1$. Choose the principal homomorphism
$\iota_{\check L}:SL_2\to\check L$ whose differentials on the upper and
lower root vectors are $e_\delta$ and $f_\delta$, respectively, and whose
restriction to the diagonal torus is $\lambda$. For $\delta\leq1$ this
homomorphism is trivial. Let
$\check P_X=\check L_X\check N_X$ be the parabolic determined by $\lambda$:
its Lie algebra is the sum of the nonnegative weight spaces,
$\check L_X=Z_{\check G}(\lambda)$, and $\operatorname{Lie}\check N_X$ is
the sum of the positive weight spaces. Explicitly,
\[
 \check L_X\cong
 \begin{cases}
 GL_{2p}\times GL_1^{q-p},&p+q\in2\mathbb Z,\\
 GL_{2p+1}\times GL_1^{q-p-1},&p+q\in2\mathbb Z+1.
 \end{cases}
\]
We will use the distinguished commuting subgroup
\[
 \check L'_X=GL_{2p}\times I_{q-p}\cong GL_{2p}.
\]
Let $V=\mathbb C^{2p}$, with standard basis
$\mathbf e_1,\ldots,\mathbf e_{2p}$ and the symplectic form characterized by
$\omega(\mathbf e_i,\mathbf e_{2p+1-i})=1$ for $1\leq i\leq p$, skew-symmetry, and zero
on all other pairs of basis vectors. The dual group of $X$ is
\[
 \check G_X=\Sp(V)\subset\check L'_X\subset\check G.
\]
The block $GL_{2p}$ commutes with $\iota_{\check L}(SL_2)$, so
\[
 \rho_X:\check G_X\times SL_2\longrightarrow\check G,\qquad
 (g,h)\longmapsto g\,\iota_{\check L}(h)
\]
is the Arthur homomorphism associated to $X$.

Denote by $\check\fg, \check\fg_X,\check\fn_X$, etc,  the Lie algebra of $\check G$, $\check G_X$, $\check N_X$, etc.
Let $\Phi_{\mathrm{tr}}:\check\fg\cong\check\fg^*$ be the $\check G$-equivariant isomorphism induced by the trace paring.
Let $e_X=d\iota_{\check L}(\begin{pmatrix}
   0& 1 \\
    0 & 0
\end{pmatrix} )$
and $y_X=d\iota_{\check L}(\begin{pmatrix}
   0& 0 \\
    1 & 0
\end{pmatrix} )$,
and $\psi_X=\Phi_{\mathrm{tr}}(y_X)\in\check\fg^*$.
Denote by $\check\fn_{X,\geq2}\subset\check\fn_X$ be the $\geq2$ weight subspace with respect to
the action of $\lambda$. Put
$\check N_{X,\geq2}=\exp(\check\fn_{X,\geq2})$.
Following \cite[Section 3.4.2]{BZSV}, we will consider the following three  Hamiltonian 
varieties:
\begin{enumerate}

\item the cotangent bundle 
$T^*\check G$ viewed as $\check G\times\check G$-Hamiltonian 
space with moment map 
$T^*\check G\cong\check G\times\check\fg^*\to\check\fg^*\times \check\fg^*$ sending 
$(g,v)\to(\text{Ad}_gv,-v)$
(here we identify $\check\fg^*$ as left invariant differential forms on $\check G$).

\item  the  symplectic representation $\check S_X$ of $\check G_X=\Sp_{2p}$ 
given by
\[
	\check S_X=
	\begin{cases}
		T^*\mathbb A^
        {2p}\ \ \ \ p=q,\\
		0\ \ \ \ \ \ \ \ \ \ \ p\neq q.
	\end{cases}
\]
where $T^*\mathbb A^
        {2p}$ is the cotangent bundle of the standard representation 
        $\mathbb A^
        {2p}$ of $\Sp_{2p}$,
viewed as $\check G_X$-Hamiltonian space 
with the standard moment map 
$\check S_X\to\check\fg_X^*=\mathfrak{sp}_{2p}^*$
sending $s\in\check S_X$ to $(v\to\frac{1}{2}\omega_{\check S_X}(vs,s))\in\mathfrak{sp}_{2p}^*$
(here $\omega_{\check S_X}$ is the canonical symplectic form on $T^*\mathbb A^{2p}$)

\item
the quotient space
$\check\fn_X/\check\fn_{X,\geq2}$
with the symplectic form
 $\omega_{\check{\mathfrak n}_X}(u,u')=\psi_X([u,u'])$
and 
the Hamiltonian $\check G_X\check N_X$-action given by 
 
 (i) the $\check G_X$-action is through the adjoint action
 
(ii) the $\check N_X$-action is through the translation action 
 $\check N_X/\check N_{X,\geq2}\cong \check\fn_X/\check \fn_{X,\geq2}$
 
 (iii) the $\check G_X$-moment map 
is given by 
\[\check\fn_X/\check \fn_{X,\geq2}\to\mathrm{Lie}(\Sp(\check\fn_X/\check \fn_{X,\geq2}))^*\to\check\fg^*_X\]
where the first map the canonical moment map on the symplectic vector space and the second map is induced by the inclusion 
$\check G_X\to \Sp(\check\fn_X/\check \fn_{X,\geq2}))$

(iv) the $\check N_X$-moment map is given by $\psi_X$-shifted moment map
 \[\check\fn_X/\check \fn_{X,\geq2}\cong (\check\fn_X/\check \fn_{X,\geq2})^*\to \check\fn_X^*\ \ \ v\to v+\psi_X|_{\check\fn_X}\]
where the first isomorphism is induced by the symplectic form.
\end{enumerate}

\begin{example}\label{even/odd}
(1) Assume $p+q$ is odd. 
Then we have $\check S_X=0$. Moreover, it follows from~\eqref{rho} that the weights with respect to co-character $2\rho_{\check L}$ 
are all even numbers and it follows that 
$\check\fn_X/\check\fn_{X,\geq2}=0$
    
(2)    
Assume $p+q$ is even. 
If $p=q$, then we have $\check S_X\cong T^*\mathbb A^{2p}$ and 
$\check\fn_X/\check\fn_{X,\geq2}=0$.
If $p\neq q$, 
then we have 
$\check S_X=0$ and a direct computation show that 
$\Sp_{2p}$-equivariant isomorphism 
\[\check\fn_X/\check\fn_{X,\geq2}\cong\check\fn_{X,1} \cong T^*\mathbb A^{2p}.\]
where $\check\fn_{X,1}$ is the weight one subspace.

\end{example}

The product space 
\[T^*\check G\times \check S_X\times\check\fn_X/\check\fn_{X,\geq2}\]
is naturally a Hamiltonain $\check G_X\check N_X$-space, 
where $\check N_X$ acts trivially on 
$\check S_X$
and $\check G_X\check N_X\subset\check G$  acts on $T^*\check G$ through the right $\check G$-action,
with moment map 
\[\mu_{\psi_X}:T^*\check G\times \check S_X\times\check\fn_X/\check\fn_{X,\geq2}\to\check\fg_X^{*}\times\check\fn_X^*\]
and 
the dual variety of $X$ is defined as the following Hamiltonian reduction  
\[\check M=\mu^{-1}_{\psi_X}(0)\sslash\check G_X\check N_X=T^*\check G\times_{\check\fg_X^{*}\times\check\fn_X^*}^{\check G_X\check N_X}(\check S_X\times\check\fn_X/\check\fn_{X,\geq2})\]
We have the residual moment map 
\[\mu:\check M\to\check\fg^*\]
coming from the left $\check G$-action on $T^*\check G$. Moreover, $\mu$ is $\mathbb G_m$-equivariant with respect to the square action on $\check\fg^*$ and the $\mathbb G_m$ on 
$\check M$ induced by
\begin{enumerate}
    \item the scalar action on $\check S_X$
    and $\check\fn_X/\check\fn_{X,\geq2}$
    \item the action on 
    $T^*\check G\cong\check G\times\check\fg^*$
    given by $t(g,v)=(g2\rho_{\check L}(t^{-1}),t^2\text{Ad}_{2\rho_{\check L}(t^{})}(v)).$
\end{enumerate}

\begin{example}\label{cotangent types}
(1) Assume $p+q$ is odd.
Then it follows from Example \eqref{even/odd}
that 
\[\check M\cong T^*_{\psi_X}(GL_{p+q}/\Sp_{2p}\check N_X)\cong GL_{p+q}\times^{\Sp_{2p}
\check N_X}(\psi_X+(\mathfrak{sp}_{2p}\oplus\check\fn_X)^\perp)\]
is isomorphic to 
the $\psi_X$-twisted cotangent bundle of $GL_{p+q}/\Sp_{2p}\check N_X$
with respect to the covector 
$\psi_X|_{\mathfrak{sp}_{2p}\oplus\check\fn_X}\in(\mathfrak{sp}^*_{2p}\oplus\check\fn_X^*)^{\Sp_{2p}\check N_X}$,
and the moment map is given by
\[\mu:\check M\to\check\fg^*,\qquad [g,v]\longmapsto\operatorname{Ad}_g v.\]
Moreover, if $q=p+1$, then
both
$\check N_X$ and $\psi_X$ are trivial, and 
\[\check M\cong T^*(GL_{2p+1}/\Sp_{2p})\cong GL_{2p+1}\times^{\Sp_{2p}}\mathfrak{sp}_{2p}^\perp\]
is isomorphic to the cotangent bundle of
$GL_{2p+1}/\Sp_{2p}$; here
$\mathfrak{sp}_{2p}^{\perp}$ is taken in $\mathfrak{gl}_{2p+1}$.

(2) Assume $p=q$. Then $\check N_X$ is trivial, and Hamiltonian induction
identifies
\[
 \check M\simeq T^*(GL_{2p}/\Sp_{2p}\times V)
 \simeq GL_{2p}\times^{\Sp_{2p}}
       (\mathfrak p_{2p}\times V\times V),
 \qquad \mathfrak p_{2p}=\mathfrak{sp}_{2p}^{\perp}\subset\mathfrak{gl}_{2p}.
\]
Use $\omega$ to identify $V^*$ with $V$, and put
$R_{u,v}(z)=\omega(v,z)u$. In the product cotangent coordinates the
moment map is
\begin{equation}\label{mu}
 \mu([g,T,u,v])=\operatorname{Ad}_g(T+R_{u,v}),
 \qquad T\in\mathfrak p_{2p}.
\end{equation}
For comparison with the trace-orthogonal coordinates in the Hamiltonian
reduction, set
\[
 \beta(u,v)=\operatorname{pr}_{\mathfrak{sp}_{2p}}R_{u,v}
          =\tfrac12(R_{u,v}+R_{v,u}),\qquad
 A=T+\operatorname{pr}_{\mathfrak p_{2p}}R_{u,v}.
\]
Then $T+R_{u,v}=A+\beta(u,v)$. Both coordinate descriptions give the
grading
$t[g,A,u,v]=[g,t^2A,tu,tv]$.
\end{example}

\subsection{Vector bundle structures}
Consider the $\check G_X\times\mathbb G_m$-representation
\[
 \check V_X=\check S_X\times
       \bigl(\check\fg_X^\perp\cap Z_{\check\fg}(e_X)\bigr).
\]
The action on $\check S_X$ is scalar multiplication, and the action on
the second factor is $x\mapsto t^2\operatorname{Ad}_{\lambda(t)}x$.
The  construction of \cite[Section 3.4.8, (3.18)--(3.19)]{BZSV}
gives an isomorphism
\begin{equation}\label{vector bundle}
 \check M\simeq\check G\times^{\check G_X}\check V_X.
\end{equation}
Here and below our associated-bundle convention is
$(g,x)\sim(gh^{-1},hx)$.

In terms of the isomorphism above the moment map is given as follows.
If $\delta=0$, 
then $\check V_X\cong \mathfrak p_{2p}\times V\times V$ and 
the moment map is described in Example~\ref{cotangent types} (2).
If $\delta>0$, then $\check S_X=0$ and we set
\[
 \mathcal Z_\delta
 =f_\delta+\bigl(\check\fg_X^\perp\cap Z_{\check\fg}(e_\delta)\bigr),
\]
where the two nilpotents occupy the final $\delta$-block. Translation by
$f_\delta$ identifies $\mathcal Z_\delta$ with $\check V_X$, and
\begin{equation}\label{slodowy moment}
 \check M\simeq\check G\times^{\check G_X}\mathcal Z_\delta,
 \qquad \mu([g,x])=\operatorname{Ad}_g x.
\end{equation}
The grading is
\begin{equation}\label{dual grading}
 t[g,x]=[g\lambda(t)^{-1},t^2\operatorname{Ad}_{\lambda(t)}x].
\end{equation}

\begin{lemma}\label{decomp}
There is a $\check G_X\times\mathbb G_m$-equivariant isomorphism
\[
 \check V_X\simeq
 \wedge^2V\times T^*V\times\mathfrak c_{GL_\delta}.
\]
The degrees on the first two factors are $2$ and $\delta+1$,
respectively. The coefficient of order $i$ of the monic polynomial
parametrized by $\mathfrak c_{GL_\delta}$ has degree $2i$.
\end{lemma}
\begin{proof}
For $\delta=0$, this is the decomposition
$\check V_X=\mathfrak p_{2p}\times T^*V$, with
$\mathfrak p_{2p}\simeq\wedge^2V$.
For $\delta>0$, the nilpotent has the block form
\[
 e_X=\begin{pmatrix}0_{2p}&0\\0&e_\delta\end{pmatrix},
 \qquad
 e_\delta=\sum_{i=1}^{\delta-1}i(\delta-i)E_{i+1,i}.
\]
Let $\varepsilon_1,\ldots,\varepsilon_\delta$ be the standard basis of
$\mathbb C^\delta$. The two off-diagonal centralizer spaces are
\[
 \begin{aligned}
 D&=\{d\in M_{2p,\delta}:de_\delta=0\}
   =\{b\varepsilon_1^t:b\in V\},\\
 C&=\{c'\in M_{\delta,2p}:e_\delta c'=0\}
   =\{\varepsilon_\delta c:c\in V^*\}.
 \end{aligned}
\]
Here $M_{a,b}$ denotes the space of $a\times b$ matrices. Thus
\begin{equation}\label{decomp of V_X}
 \begin{aligned}
 \check V_X
 &=\check\fg_X^\perp\cap Z_{\check\fg}(e_X)
   =\begin{pmatrix}
     \mathfrak p_{2p}&D\\ C&Z_{\mathfrak{gl}_\delta}(e_\delta)
    \end{pmatrix}\\
 &=\left\{
 \begin{pmatrix}
 A&b\varepsilon_1^t\\
 \varepsilon_\delta c&s
 \end{pmatrix}
 : A\in\mathfrak p_{2p},\ b\in V,\ c\in V^*,\
   s\in Z_{\mathfrak{gl}_\delta}(e_\delta)
 \right\}.
 \end{aligned}
\end{equation}
Thus the matrix $x$ in \eqref{slodowy moment} is
\begin{equation}\label{slodowy blocks}
 x(A,b,c,s)=
 \begin{pmatrix}
 A&b\varepsilon_1^t\\
 \varepsilon_\delta c&f_\delta+s
 \end{pmatrix}.
\end{equation}
Since $e_\delta$ is regular nilpotent, its Lie algebra centralizer
$\mathcal S_\delta$ has basis
$I_\delta,e_\delta,\ldots,e_\delta^{\delta-1}$.
The grading gives degrees $2$ on $A$ and $\delta+1$ on $b,c$.
The coefficient of $e_\delta^{i-1}$ has degree $2i$.
The Kostant section
$f_\delta+Z_{\mathfrak{gl}_\delta}(e_\delta)\simeq\mathfrak c_{GL_\delta}$
therefore identifies its grading with the characteristic-polynomial
degrees $2,4,\ldots,2\delta$, proving the assertion.
\end{proof}

\begin{proposition}\label{induced}
There is an isomorphism of varieties
\[
 \check M\simeq GL_{p+q}\times^{GL_{2p}}
 T^*(GL_{2p}/\Sp_{2p}\times V)\times\mathfrak c_{GL_\delta}.
\]
It is $\check G\times\mathbb G_m$-equivariant for the action transported
from \eqref{dual grading} and the fiber degrees in Lemma~\ref{decomp}.
\end{proposition}
\begin{proof}
Combining \eqref{vector bundle}, Lemma~\ref{decomp}
and Example~\ref{cotangent types}, we have 
the desired isomorphism
\[
 \check M\simeq
 GL_{p+q}\times^{\Sp_{2p}}
       (\mathfrak p_{2p}\times T^*V)\times\mathfrak c_{GL_\delta}\simeq GL_{p+q}\times^{GL_{2p}}
 T^*(GL_{2p}/\Sp_{2p}\times V)\times\mathfrak c_{GL_\delta}.
\]

\end{proof}

\subsection{Invariant quotients
and exotic nilpotent cones}\label{invariant quotients}
Put $\delta=q-p$ and $V=\mathbb A^{2p}$. Under the symplectic
identification $\wedge^2V\simeq\mathfrak p_{2p}$, the space
$\mathfrak p_{2p}=\mathfrak{sp}_{2p}^{\perp}$ consists of the
symplectic-self-adjoint endomorphisms. Explicitly,
\[
 \mathfrak p_{2p}
 =\left\{
   \begin{pmatrix}X&Y\\ Z&r(X)\end{pmatrix}
   : X,Y,Z\in\mathfrak{gl}_p,\quad r(Y)=-Y,\quad r(Z)=-Z
  \right\}.
\]
Here $r(X)=J_pX^tJ_p$ is reflection across the anti-diagonal, with
$J_p$ the anti-diagonal permutation matrix.
The invariant ring for the
diagonal action of $\Sp_{2p}$ on
$\mathfrak p_{2p}\times V\times V$ is polynomial, with generators
$b_1,\ldots,b_p,c_0,\ldots,c_{p-1}$ characterized by
\[
 \det(z-T)=P(z)^2,\quad
 P(z)=z^p+b_1z^{p-1}+\cdots+b_p,\qquad
 c_j=-\omega(T^ju,v).
\]
Here $(T,u,v)\in\mathfrak p_{2p}\times V\times V$.
The quotient map is equidimensional; see \cite{S} and \cite{Kn2}.

Recall that $\fc_{GL_m}=\mathfrak{gl}_m\sslash GL_m$ is the coefficient
space of monic polynomials of degree $m$, with origin $z^m$.
Using the trace identification, let
$\chi_m:\mathfrak{gl}_m^*\to\fc_{GL_m}$ be the
characteristic-polynomial quotient map.
We use the coordinates $c_0,\ldots,c_{p-1}$ for a second copy of
$\fc_{GL_p}$ by associating to them the polynomial
\[
 C(z)=z^p+c_0z^{p-1}+\cdots+c_{p-1}.
\]
The final factor $\fc_{GL_\delta}$ has coordinates $d_1,\ldots,d_\delta$,
with
\[
 D(z)=z^\delta+d_1z^{\delta-1}+\cdots+d_\delta;
\]
for $\delta=0$ this factor is a point and $D=1$.
Lemma~\ref{decomp} and Proposition~\ref{induced} give the natural product
presentations
\[
 \begin{aligned}
 (\mathfrak p_{2p}\times V\times V)\sslash\Sp_{2p}
   &\simeq\fc_{GL_p}\times\fc_{GL_p},\ \ \ 
 \check M\sslash\check G
   &\simeq\fc_{GL_p}\times\fc_{GL_p}\times\fc_{GL_\delta}.
 \end{aligned}
\]
In particular,
\[
 \mathcal O(\check M)^{\check G}
 =\mathbb C[b_1,\ldots,b_p,c_0,\ldots,c_{p-1},d_1,\ldots,d_\delta].
\]
The grading transported from $\check M$ is
\[
 \deg b_i=2i,\qquad \deg c_j=2(j+\delta+1),\qquad \deg d_i=2i.
\]
Denote these quotient maps by 
\[\chi_{p,p}:\mathfrak p_{2p}\times V\times V\to\fc_{GL_p}\times\fc_{GL_p},\ \ \ \ 
\chi_{\check M}:\check M\to\fc_{GL_p}\times\fc_{GL_p}\times\fc_{GL_\delta}
.\]
 We next construct
sections of $\chi_{p,p}$ and $\chi_{\check M}$
using the companion-matrix description.
For $b=(b_1,\ldots,b_p)$, take the companion matrix
\[
 \kappa_p(b)=
 \begin{pmatrix}
  0&0&\cdots&0&-b_p\\
  1&0&\cdots&0&-b_{p-1}\\
  0&1&\ddots&\vdots&\vdots\\
  \vdots&\ddots&\ddots&0&-b_2\\
  0&\cdots&0&1&-b_1
 \end{pmatrix}.
\]

\begin{lemma}\label{section}
\begin{enumerate}
\item The map
\[
 \kappa_{p,p}(b,c)=
 \left(
 \begin{pmatrix}\kappa_p(b)&0\\0&r(\kappa_p(b))\end{pmatrix},\,
 \mathbf e_1,\,(0,\ldots,0,-c_{p-1},\ldots,-c_0)^t
 \right)
\]
is an algebraic section of $\chi_{p,p}$.
\item In the product fiber coordinates above, the map
$\kappa_{\check M}=\kappa_{p,p}\times\mathrm{id}_{\fc_{GL_\delta}}$,
\[
 \begin{aligned}
 \fc_{GL_p}\times\fc_{GL_p}\times\fc_{GL_\delta}
 &\xrightarrow{\ \kappa_{\check M}\ }\check V_X\\
 &\hookrightarrow GL_{p+q}\times^{\Sp_{2p}}\check V_X\simeq\check M,
 \end{aligned}
\]
defines an algebraic section of $\chi_{\check M}$, with the final
factor omitted when $\delta=0$.
\end{enumerate}
\end{lemma}
\begin{proof}
Write $\kappa_{p,p}(b,c)=(T,u,v)$. The characteristic polynomial of
$T$ is $(z^p+b_1z^{p-1}+\cdots+b_p)^2$. Since
$T^ju=\mathbf e_{j+1}$ for $0\leq j<p$, we also have
$-\omega(T^ju,v)=c_j$. This proves (1); part (2) follows immediately
from part (1).
\end{proof}

The moment map $\mu:\check M\to\mathfrak{gl}_{p+q}^*$ is
$GL_{p+q}$-equivariant and hence gives rise to the following
commutative diagram:
\begin{equation}\label{moment quotient}
\begin{tikzcd}
\check M\ar[r,"\mu"]\ar[d,"\chi_{\check M}",swap]
 &\mathfrak{gl}_{p+q}^*\ar[d,"\chi_N"]\\
\fc_{GL_p}\times\fc_{GL_p}\times\fc_{GL_\delta}
 \ar[r,"\bar\mu",swap]
 &\fc_{GL_{p+q}}.
\end{tikzcd}
\end{equation}
Let $\mathcal N_{\check\fg^*}=\chi_N^{-1}(0)$ be the nilpotent
cone for $\mathfrak{gl}_{p+q}^*$, and put
$\mathcal N'_{\Sp_{2p}}=\chi_{p,p}^{-1}(0)$, the exotic null-cone
in this representation; see \cite{Ka}. It follows from
Proposition~\ref{induced} that
\[
 \chi_{\check M}^{-1}(0)
 \simeq GL_{p+q}\times^{\Sp_{2p}}\mathcal N'_{\Sp_{2p}}.
\]
Consider the preimage $\mu^{-1}(\mathcal N_{\check\fg^*})\subset\check M$
of the nilpotent cone under the moment map.

\begin{lemma}\label{exotic nilpotent cone}
The natural closed embedding
\[
 GL_{p+q}\times^{\Sp_{2p}}\mathcal N'_{\Sp_{2p}}
 \simeq\chi_{\check M}^{-1}(0)
 \hookrightarrow\mu^{-1}(\mathcal N_{\check\fg^*})
\]
induces a bijection on $\mathbb C$-points.
\end{lemma}
\begin{proof}
Denote by $\bar\mu$ the lower arrow in~\eqref{moment quotient}.
In view of the commutativity of that diagram, it suffices to show
that the $\mathbb C$-points of $\bar\mu^{-1}(0)$ consist only of the
origin. By Lemma~\ref{section}, we have
\[
 \bar\mu(b,c,d)=\chi_N\circ\mu\circ\kappa_{\check M}(b,c,d),
\]
and $\bar\mu(b,c,d)=0$ if and only if the characteristic polynomial of
$\mu(\kappa_{\check M}(b,c,d))$ is $z^{p+q}$. Thus we reduce to
checking that
\[
 \det\bigl(z-\mu(\kappa_{\check M}(b,c,d))\bigr)=z^{p+q}
 \quad\Longleftrightarrow\quad (b,c,d)=0.
\]

Write $\kappa_{\check M}(b,c,d)=(T,u,v,d)$, put
$K_b=\kappa_p(b)$, and set
\[
 R(c)=
 \begin{pmatrix}
 c_0&\cdots&c_{p-1}\\
 0&\cdots&0\\
 \vdots&&\vdots\\
 0&\cdots&0
 \end{pmatrix}.
\]
Then equations~\eqref{mu} and~\eqref{slodowy blocks} imply
\[
 \mu(\kappa_{\check M}(b,c,d))=
 \begin{cases}
 \begin{pmatrix}K_b+R(c)&0\\0&r(K_b)\end{pmatrix},&p=q,\\[8pt]
 \begin{pmatrix}
   K_b&0&E_{p,\delta}\\
   0&r(K_b)&0\\
   F_{\delta,p}(c)&0&f_\delta+s(d)
  \end{pmatrix},&p<q.
 \end{cases}
\]
Here $E_{p,\delta}$ is $p\times\delta$ with its only nonzero entry
$1$ in position $(1,1)$, and $F_{\delta,p}(c)$ is $\delta\times p$
with last row $(c_0,\ldots,c_{p-1})$ and all other rows zero.
For $\delta>0$, $s(d)\in Z_{\mathfrak{gl}_\delta}(e_\delta)$ is the unique element
such that $\det(z-f_\delta-s(d))=D(z)$.

First, in both cases $r(K_b)$ is a direct summand of the moment matrix.
If the moment matrix is nilpotent, then $r(K_b)$ is nilpotent. Since
\[
 \det(z-r(K_b))=z^p+b_1z^{p-1}+\cdots+b_p,
\]
this forces $b=0$.

With $b=0$, expansion of the displayed determinants gives
\[
 \det\bigl(z-\mu(\kappa_{\check M}(0,c,d))\bigr)
 =z^p\left(z^pD(z)-\sum_{j=0}^{p-1}c_jz^{p-1-j}\right),
 \qquad D=1\text{ if }\delta=0.
\]
Thus nilpotency is equivalent to
\[
 z^pD(z)-\sum_{j=0}^{p-1}c_jz^{p-1-j}=z^q.
\]
The term $z^pD(z)$ has no powers of degree below $p$, so comparison
of those coefficients gives $c=0$. We then obtain
$D(z)=z^{q-p}=z^\delta$, hence $d=0$.
Conversely, the same formula gives $z^{p+q}$ when $(b,c,d)=0$.
This proves the lemma.
\end{proof}

\subsection{Whittaker cotangent bundles}\label{sect:whit cot}
Let $\check K=GL_p\times GL_q$ and let $\check P$ be the lower
triangular parabolic with Levi $\check K$.
Put
\[
 f_{p,q}=\operatorname{diag}(f_p,f_q),\qquad
 \tau_{\check K}(t)=\operatorname{diag}(\tau_p(t),\tau_q(t)),
 \qquad \tau_{\check G}(t)=\tau_N(t).
\]
The additive character $\psi_{\check K}:\check U\to\mathbb G_a$ has
differential $u\mapsto\operatorname{tr}(f_{p,q}u)$: it is $1$ on the pinned
simple root vectors in $\check K$ and $0$ on the simple root vectors
outside $\check K$. For the right action $g\mapsto gu^{-1}$,
the Whittaker cotangent bundle is the Hamiltonian reduction
at $-d\psi_{\check K}$:
\[T^*_{\psi_{\check K}}(\check G/\check U)
 \simeq\check G\times^{\check U}(f_{p,q}+\check{\mathfrak b})
 \]
 with moment map sending $[g,x]$ to $\operatorname{Ad}_g x$.
The grading relevant to the dual variety is defined by
\begin{equation}\label{Whittaker cocharacter}
 \gamma(t)=\tau_{\check K}(t)^2\tau_{\check G}(t)^{-1}
 =\tau_{\check K}(t)\operatorname{diag}(t^qI_p,t^{-p}I_q),
\end{equation}
and
\begin{equation}\label{Whittaker grading}
 t[g,x]=[g\gamma(t)^{-1},t^2\operatorname{Ad}_{\gamma(t)}x].
\end{equation}

\subsection{Regular centralizers}\label{sect: reg cent}
We recall some facts and constructions related to the group scheme of
regular centralizers (see, e.g., \cite{BF,Kn1,Ngo,Gi,L}).
Let $J_{\check G}\to\fc_{\check G}$ (resp.
$J_{\check K}\to\fc_{\check K}$) be the group scheme of regular
centralizers of $\check G$ (resp. $\check K$).
Choose the homogeneous Kostant sections
$\kappa_{\check G}:\fc_{\check G}\to\check\fg^*$ and
$\kappa_{\check K}:\fc_{\check K}\to\check\fk^*$, whose images on a block
of size $m$ are $f_m+Z_{\mathfrak{gl}_m}(e_m)$. They give rise to closed
embeddings
\[
 J_{\check G}\longrightarrow\fc_{\check G}\times\check G,
 \qquad
 J_{\check K}\longrightarrow\fc_{\check K}\times\check K.
\]
The fiber over $x$ is the centralizer of $\kappa_{\check G}(x)$,
respectively of $\kappa_{\check K}(x)$. Moreover, we have natural
isomorphisms
\[
 \operatorname{Lie}J_{\check G}\simeq T^*\fc_{\check G},
 \qquad
 \operatorname{Lie}J_{\check K}\simeq T^*\fc_{\check K}.
\]
For $A=\check G,\check K$, the chosen sections satisfy
\[
 \kappa_A(t^2x)=t^2\operatorname{Ad}_{\tau_A(t)}\kappa_A(x),
\]
where $t^2x$ denotes square dilation on the Chevalley quotient.

The embedding $\check K\to\check G$ gives rise to the map between
Chevalley quotients
\[
 \pi:\fc_{\check K}=\check\fk^*\sslash\check K
 \longrightarrow\check\fg^*\sslash\check G=\fc_{\check G},
\]
and hence to the cotangent map
\[
 \pi^*\operatorname{Lie}J_{\check G}
 \simeq\pi^*T^*\fc_{\check G}
 \longrightarrow T^*\fc_{\check K}
 \simeq\operatorname{Lie}J_{\check K}.
\]
As in \cite{L} and \cite[Proposition 5.4.1]{Gi}, this map is the
differential of a homomorphism
\begin{equation}\label{maps between J}
 \alpha:\pi^*J_{\check G}\longrightarrow J_{\check K}
\end{equation}
of group schemes over $\fc_{\check K}$. Note that $\pi$ is generically
\'etale and the map~\eqref{maps between J} is generically an isomorphism.

Using the above homomorphism, we can form the quotient
\[
(\check G\times J_{\check K})/\pi^*J_{\check G}
 =\bigl((\check G\times\fc_{\check K})
          \times_{\fc_{\check K}}J_{\check K}\bigr)/\pi^*J_{\check G}
\]
by the group scheme $\pi^*J_{\check G}$ over $\fc_{\check K}$.
Its action on the two factors is through the Kostant-section embedding
and the homomorphism~\eqref{maps between J}, respectively; explicitly,
\[
 (g,x,y)\cdot j=(gj^{-1},x,\alpha(j)y).
\]
We have a natural $\check G\times\mathbb G_m$-action on this quotient,
given by
\[
 (h,t)[g,x,y]
 =[hg\tau_{\check G}(t)^{-1},t^2x,
                         \operatorname{Ad}_{\tau_{\check K}(t)}y],
\]
and a $\check G\times\mathbb G_m$-equivariant map
\begin{equation}\label{open moment map}
(\check G\times J_{\check K})/\pi^*J_{\check G}\longrightarrow\check\fg^*,\qquad
 [g,x,y]\longmapsto\operatorname{Ad}_g\kappa_{\check G}(\pi(x)).
\end{equation}
The morphism~\eqref{open moment map} extends to a
$\check G\times\mathbb G_m$-equivariant map
$\overline{(\check G\times J_{\check K})/\pi^*J_{\check G}}\to\check\fg^*$.

\begin{proposition}\label{Wh=J}
There is a $\check G\times\mathbb G_m$-equivariant isomorphism of affine
varieties over $\check\fg^*$
\[
 \overline{T^*_{\psi_{\check K}}(\check G/\check U)}
 \simeq\overline{(\check G\times J_{\check K})/\pi^*J_{\check G}},
\]
with the grading~\eqref{Whittaker grading} on the left.
\end{proposition}
\begin{proof}
According to \cite[Theorem 1.1]{Ga}, there is a
$\check G\times\check K$-equivariant isomorphism over $\check\fg^*$
\[
 \overline{T^*(\check G/\check U_{\check P})}
 \simeq
 \overline{(\check G\times\check K\times\fc_{\check K})/
                     \pi^*J_{\check G}},
\]
where $\check U_{\check P}\subset\check P$ is the unipotent radical.
Taking Whittaker reduction with respect to $\check K$ and
$\psi_{\check K}$, we obtain the desired isomorphism
\[
 \begin{aligned}
 \overline{T^*_{\psi_{\check K}}(\check G/\check U)}
 &\simeq
 \overline{T^*(\check G/\check U_{\check P})
                \sslash_{-d\psi_{\check K}}\check U_{\check K}}\\
 &\simeq
 \overline{(\check G\times J_{\check K})/\pi^*J_{\check G}}.
 \end{aligned}
\]
Here $\check U_{\check K}=\check U\cap\check K$, the differential
is restricted to $\operatorname{Lie}\check U_{\check K}$, and the sign is that
of the right action $g\mapsto gh^{-1}$.
The first isomorphism follows from \cite[Lemma 3.1.10]{M}; the second
is the Whittaker reduction of Gannon's theorem, as in
\cite[Corollary 2.8]{Ga}.
\end{proof}

\subsection{Dual varieties as affine closures}
Here is the main result of the section, providing an alternative
description of the dual variety $\check M$ as an affine closure of
$T^*_{\psi_{\check K}}(\check G/\check U)$.

\begin{theorem}\label{affine closures}
There is a $\check G\times\mathbb G_m$-equivariant isomorphism of affine
varieties over $\check\fg^*$
\[
 \check M\simeq
 \overline{T^*_{\psi_{\check K}}(\check G/\check U)}.
\]
\end{theorem}

The rest of the section is devoted to the proof of
Theorem~\ref{affine closures}. The key ingredients include a
generalization of Kostant's theorem, see
Section~\ref{generalized Kostant thm}, and a Lagrangian correspondence
between $T^*_{\psi_{\check K}}(\check G/\check U)$ and $\check M$, see
Section~\ref{Lag}.

Since the Hamiltonian dual variety $\check M$ is smooth, as a corollary
of Theorem~\ref{affine closures} and Proposition~\ref{Wh=J} we obtain
the following smoothness of the affine closures.
\begin{corollary}\label{cor: aff clos}
The affine closures
\[
 \overline{T^*_{\psi_{\check K}}(\check G/\check U)}
 \simeq\overline{(\check G\times J_{\check K})/\pi^*J_{\check G}}
\]
are smooth.
\end{corollary}

\subsection{Generalized Kostant theorem}\label{generalized Kostant thm}
For $m\geq1$, let $\check B_m\subseteq GL_m$ be the lower triangular Borel,
let $\check U_m$ be its unipotent radical, and write $\check{\mathfrak b}_m$ and
$\check{\mathfrak u}_m$ for their Lie algebras. 
Consider the  regular nilpotent matrices
\[
 e_m=\begin{pmatrix}
 0&&&&\\
 m-1&0&&&\\
 &2(m-2)&0&&\\
 &&\ddots&\ddots&\\
 &&&m-1&0
 \end{pmatrix},\qquad
 f_m=\begin{pmatrix}
 0&1&&&\\
 &0&1&&\\
 &&\ddots&\ddots&\\
 &&&0&1\\
 &&&&0
 \end{pmatrix}.
\]
Put
\[
 \tau_m(t)=\operatorname{diag}(t^{1-m},t^{3-m},\ldots,t^{m-1}).
\]
Thus $(e_m,d\tau_m(1),f_m)$ is a principal $\mathfrak{sl}_2$-triple.
The celebrated result of 
Kostant \cite{K} asserts that $\check U_m$ acts freely on $\check{\mathfrak b}_m+f_m$ and
provides a transverse slice to the $\check U_m$-orbits. More precisely, let
$\mathcal S_m=Z_{\mathfrak{gl}_m}(e_m)\subseteq\check{\mathfrak b}_m$.
Then the action map
\[
 \check U_m\times(\mathcal S_m+f_m)\longrightarrow\check{\mathfrak b}_m+f_m,
 \qquad(u,x)\longmapsto\operatorname{Ad}_u(x),
\]
is a $U_m$-equivariant isomorphism. 

The main result of this section is an analog
of Kostant's theorem in which $f_m$ is replaced by a nilpotent element
with two Jordan blocks. Namely, consider the following nilpotent elements
\[
 e_{p,q}=\begin{pmatrix}e_p&0\\0&e_q\end{pmatrix},\qquad
 f_{p,q}=\begin{pmatrix}f_p&0\\0&f_q\end{pmatrix}.
\]
The second matrix is the trace representative of $d\psi_{\check K}$
used in Section~\ref{sect:whit cot}.
We have a natural adjoint action of $\check U$ on
$\check\fb+f_{p,q}$ but, unlike the setting of Kostant theorem, this action is typically not free. We shall introduce a subgroup $\check U_X\subseteq \check U$ and a
$\check U_X$-invariant affine subspace
$\mathcal S_{p,q}+f_{p,q}\subseteq\check\fb+f_{p,q}$.
The generalized Kostant theorem, see Theorem \ref{whittaker-main} below, says that 
the adjoint action map factors through an isomorphism
\begin{equation}\label{GK}
    \check U\times^{\check U_X}(\mathcal S_{p,q}+f_{p,q})
       \simeq\check\fb+f_{p,q}.
\end{equation}

\subsubsection{Symplectic linear algebra}
We record some linear algebra results needed for the proof of~\eqref{GK}.
We focus on the case $p=q$ and write $\check G=GL_{2p}(\bC)$. Let $V=\bC^{2p}$, and let
$\mathbf e_1,\ldots,\mathbf e_{2p}$ be its standard basis. Equip $V$ with the symplectic form
$\omega(\mathbf e_i,\mathbf e_{2p+1-i})=1$ for $1\leq i\leq p$, the remaining pairings
being determined by alternatingness and zero. Write
\[
 \Omega=\begin{pmatrix}0&J_p\\-J_p&0\end{pmatrix},\qquad
 \check G_X=\Sp_{2p}=\check G^{\theta_X},\qquad
 \theta_X(g)=\Omega^{-1}(g^{\mathrm t})^{-1}\Omega,
\]
where $J_p$ is the reversing matrix of size $p$. The induced involution on
$\check\fg$ gives $\check\fg=\check\fg_X\oplus\mathfrak p_{2p}$, with
$\check\fg_X=\check\fg^{\theta_X}$ and $\mathfrak p_{2p}=\check\fg^{-\theta_X}$.
If $r(A)=J_pA^{\mathrm t}J_p$, then
\[
 \theta_X\begin{pmatrix}A&B\\C&D\end{pmatrix}
 =\begin{pmatrix}-r(D)&r(B)\\r(C)&-r(A)\end{pmatrix}.
\]
Let $\check B$ be the lower triangular Borel, $\check U$ its unipotent radical, and
$\check T$ the diagonal torus. These subgroups are $\theta_X$-stable. Put
$\check U_X=\check U\cap\check G_X=\check U^{\theta_X}$ and
$\check{\mathfrak u}_X=\operatorname{Lie}(\check U_X)$. Let $\check B^-$
be the opposite Borel and $\check U^-$ its unipotent radical. Write
$V=W_1\oplus W_2$ for the two coordinate subspaces of dimension $p$.
Let $\check P\supseteq \check B$ be the parabolic stabilizing $W_2$, let $\check U_{\check P}$ be its
unipotent radical, and let $\check K=GL_p\times GL_p$ be its Levi subgroup
preserving $W_1$ and $W_2$.

Using the weighted matrices fixed above, define
\[
 e=\begin{pmatrix}e_p&0\\0&e_p\end{pmatrix},\qquad
 f=\begin{pmatrix}f_p&0\\0&f_p\end{pmatrix}
   =f_{p,p},\qquad
 h=[e,f].
\]
Then $(e,h,f)$ is a principal $\mathfrak{sl}_2$-triple in
$\check\fk=\operatorname{Lie}(\check K)$. The element $h$ is the differential
of the cocharacter $\operatorname{diag}(\tau_p,\tau_p)$, which is the
sum of the positive coroots of $\check K$. The block formula for $\theta_X$ shows that $e,f\in\mathfrak p_{2p}$ and hence
$h\in\check\fg_X$. Consequently,
$[f,\check{\mathfrak u}\cap\mathfrak p_{2p}]\subseteq\check\fb\cap\check\fg_X
=\check\fb^{\theta_X}$.

\Needspace{5\baselineskip}
\begin{lemma}\label{adf injectivity}
The map $\operatorname{ad}_f:\check{\mathfrak u}\cap\mathfrak p_{2p}\to
\check\fb^{\theta_X}$ is injective.
\end{lemma}
\begin{proof}
If $x\in\check{\mathfrak u}\cap\mathfrak p_{2p}$ satisfies $[f,x]=0$, then its image
under the projection $\pi:\operatorname{Lie}(\check P)\twoheadrightarrow\check\fk$
does the same. But $\pi(x)$ belongs to $\check{\mathfrak u}\cap\check\fk$
and $f\in\check{\mathfrak u}^-\cap\check\fk$ is a regular nilpotent in
$\check\fk$, so $\pi(x)=0$. Therefore $x\in\operatorname{Lie}(\check U_{\check P})$.

We have a $\check P$-equivariant isomorphism
$\operatorname{Lie}(\check U_{\check P})\simeq\operatorname{Hom}(V/W_2,W_2)$ and therefore an
$\check K$-equivariant isomorphism
$\operatorname{Lie}(\check U_{\check P})\simeq\operatorname{Hom}(W_1,W_2)$. Let
$\check K_\Delta\subseteq \check K$ be the diagonal subgroup. Then $W_1\simeq W_2$
as $\check K_\Delta$-modules, and $\operatorname{Lie}(\check U_{\check P})$ is isomorphic to the adjoint
representation of $\check K_\Delta$. Observe that $f\in\check\fk_\Delta$
is principal nilpotent in this reductive subalgebra. Hence its centralizer
is spanned by its powers. Thus the kernel of $\operatorname{ad}_f$ on
$\operatorname{Lie}(\check U_{\check P})$ consists of matrices
\[
 x=\begin{pmatrix}0&0\\C&0\end{pmatrix},\qquad f_p C=C f_p,
\]
where $C\in\bC[f_p]$ is upper triangular and constant along each diagonal.
But then $r(C)=J_pC^{\mathrm t}J_p=C$, so $x$ is $\theta_X$-invariant.
Since $x\in\mathfrak p_{2p}$, we conclude that $x=0$.
\end{proof}

Now we will construct a complement to the image
$\operatorname{ad}_f(\check{\mathfrak u}\cap\mathfrak p_{2p})$ in
$\check\fb^{\theta_X}$. We have the moment map
$V\times V^*\to\check\fg_X^*$. Using the trace form, we identify
$\check\fg_X$ with $\check\fg_X^*$. We also use
$w\mapsto\omega(w,-)$ to make the $\check G_X$-equivariant identification
$V\simeq V^*$. We obtain a $\check G_X$-equivariant symmetric bilinear pairing
\begin{equation}\label{eq:slice-beta}
 \beta(v,w)\in\check\fg_X,\qquad
 \beta(v,w)(z)=\tfrac12\bigl(\omega(w,z)v+\omega(v,z)w\bigr).
\end{equation}
We define a map $i:V\to\check\fg_X$ by $i(w)=\beta(\mathbf e_{2p},w)$.
Since $\mathbf e_{2p}$ is $\check U$-invariant, $i$ is $\check U_X$-equivariant.
The map $i$ is injective, because $\beta$ identifies
$\operatorname{Sym}^2V$ with $\check\fg_X$. Let $\mathfrak v=i(V)$
be its image. It is a $\check U_X$-stable subspace of $\check\fg_X$ and
consists of the $\theta_X$-fixed matrices supported in the first column
and last row. In particular, $\mathfrak v\subseteq\check\fb^{\theta_X}$.

\begin{lemma}\label{slice-intersection}
$\mathfrak v\cap[f,\check{\mathfrak u}\cap\mathfrak p_{2p}]=0$.
\end{lemma}
\begin{proof}
Write
\[
x=\begin{pmatrix}A&0\\C&r(A)\end{pmatrix}\in\check{\mathfrak u}\cap\mathfrak p_{2p},
 \]
where $A$ is  strictly lower triangular
and $r(C)=-C$.
Suppose $[f,x]\in\mathfrak v$. Its upper-left block is supported in
the first column, say $[f_p,A]=d\varepsilon_1^{\mathrm t}$, where
$\varepsilon_1=(1,0,\ldots,0)^{\mathrm t}\in\mathbb C^p$. For $0\leq k<p$,
\[
 0=\operatorname{tr}(f_p^k[f_p,A])=\varepsilon_1^{\mathrm t}f_p^kd=d_{k+1}.
\]
Hence $[f_p,A]=0$ and $A=0$. The lower-left block $[f_p,C]$ is supported in
its first column and last row, so
$C_{a+1,b}=C_{a,b-1}$ for $a<p$ and $b>1$.
Thus $C$ is constant along each diagonal, which implies $r(C)=C$.
Together with $r(C)=-C$ this gives $C=0$.
\end{proof}

\begin{corollary}\label{btheta decomposition}
$\check\fb^{\theta_X}=\mathfrak v\oplus[f,\check{\mathfrak u}\cap\mathfrak p_{2p}]$.
\end{corollary}
\begin{proof}
By Lemma~\ref{slice-intersection}, it now suffices to count dimensions.
On one hand,
\[
 \dim\check\fb^{\theta_X}
 =\frac{\dim\check\fg_X+\operatorname{rank}\check\fg_X}{2}
 =\frac{p(2p+1)+p}{2}=p^2+p.
\]
On the other hand, by the injectivity of $\operatorname{ad}_f$ on
$\check{\mathfrak u}\cap\mathfrak p_{2p}$ and the injectivity of $i$, we have
$\dim\mathfrak v=2p$ and
$\dim[f,\check{\mathfrak u}\cap\mathfrak p_{2p}]=\dim(\check{\mathfrak u}\cap\mathfrak p_{2p})$.
In the block form of Lemma~\ref{slice-intersection}, the strictly lower
triangular block $A$ contributes $p(p-1)/2$ parameters. The condition
$r(C)=-C$ says that $J_pC$ is skew-symmetric, so $C$ contributes
another $p(p-1)/2$ parameters. Thus
\[
 \dim(\check{\mathfrak u}\cap\mathfrak p_{2p})=p(p-1),\qquad
 2p+p(p-1)=p^2+p.\qedhere
\]
\end{proof}

\subsubsection{Construction of the slice}
Let $\delta=q-p$. In what follows, we need to make use of the Levi
subgroup and subalgebra given by the block embeddings
\[
 \begin{gathered}
 GL_{2p}\times GL_\delta=
 \begin{pmatrix}GL_{2p}&0\\0&GL_\delta\end{pmatrix}\subseteq \check G,\\[4pt]
 \mathfrak{gl}_{2p}\oplus\mathfrak{gl}_\delta=
 \begin{pmatrix}\mathfrak{gl}_{2p}&0\\0&\mathfrak{gl}_\delta\end{pmatrix}
 \subseteq\mathfrak{gl}_N,
 \end{gathered}
\]
with the $\delta$-block omitted when $\delta=0$.
Let $\check G_X=\Sp_{2p}\subseteq GL_{2p}$ and $\check\fg_X=\operatorname{Lie}(\check G_X)$.
We use the involution $\theta_X$ of $GL_{2p}$ defining $\check G_X=GL_{2p}^{\theta_X}$
and the Cartan decomposition
$\mathfrak{gl}_{2p}=\check\fg_X\oplus\mathfrak p_{2p}$ introduced above.
Here $\check U_{2p}=\check U\cap\check L'_X$, and $\check{\mathfrak b}_{2p}$
is the lower triangular subalgebra of the $2p$-block. The subgroup
$\check U_X\subseteq\check U$ is defined by
\[
 \check U_X=\check U_{2p}^{\theta_X}\subseteq \check U,\qquad
 u\longmapsto\begin{pmatrix}u&0\\0&I_\delta\end{pmatrix}.
\]

The construction of $\mathcal S_{p,q}$ is uniform for $p\leq q$.
Put
\[
 \mathcal S^0_{p,p}=\check{\mathfrak b}_{2p}\cap\mathfrak p_{2p},\qquad
 \mathcal S_\delta=Z_{\mathfrak{gl}_\delta}(e_\delta),
\]
where $\mathcal S_\delta$ is defined for $\delta>0$ and its summand is
omitted when $\delta=0$. We regard these as subspaces of
$\mathfrak{gl}_N$ through the block embeddings above. Let
\[
 V^0_{p,q}=\{x\in\mathfrak{gl}_N:x_{ab}=0
                         \text{ if }a<N\text{ or }b>2p\}.
\]
We define
\[
 \mathcal S_{p,q}=\mathcal S^0_{p,p}\oplus\mathcal S_\delta\oplus V^0_{p,q},
 \qquad \Sigma_{p,q}=f_{p,q}+\mathcal S_{p,q}.
\]

When $\delta=0$, the slice has the following alternative description. Put
$\operatorname{Av}_{\theta_X}=(1+\theta_X)/2$. This projection identifies
$V^0_{p,p}$ with $\mathfrak v$: its image lies in $\mathfrak v$, and it
is injective on last-row matrices, because $\theta_X$ carries the last
row to the first column, whose intersection with that row is spanned
by the $\theta_X$-fixed matrix $E_{2p,1}$. Since both spaces have
dimension $2p$, the image is all of $\mathfrak v$. For $w\in V^0_{p,p}$,
we have
\[
 w-\operatorname{Av}_{\theta_X}(w)
 =\tfrac12(w-\theta_X(w))\in\mathcal S^0_{p,p}.
\]
The same injectivity gives $\mathcal S^0_{p,p}\cap V^0_{p,p}=0$.
Consequently,
\begin{equation}\label{S}
    \mathcal S_{p,p}=\mathcal S^0_{p,p}\oplus V^0_{p,p}
 =\mathcal S^0_{p,p}\oplus\mathfrak v.
\end{equation}
 For $\delta>0$, let $\varepsilon_1,\ldots,\varepsilon_\delta$ be the
standard basis of $\bC^\delta$ and set $E=\mathbf e_{2p}\varepsilon_1^{\mathrm t}$.
The matrices in $\Sigma_{p,q}$ have the form
\begin{equation}\label{eq:general-slice-matrix}
 \begin{pmatrix}A&E\\\varepsilon_\delta c&f_\delta+s\end{pmatrix},
 \quad A\in f_{p,p}+\mathcal S^0_{p,p},\quad
 c\in(\bC^{2p})^*,\quad s\in\mathcal S_\delta.
\end{equation}
Since $\check U_X$ fixes $\mathbf e_{2p}$, this description shows that
$\Sigma_{p,q}$ is $\check U_X$-stable when $\delta>0$. For $\delta=0$,
this follows from the stability of $\mathfrak v$ and
$\operatorname{Ad}_{\check U_X}(f_{p,p})-f_{p,p}\subseteq\mathcal S^0_{p,p}$.

\begin{example}\label{example:slices}
Consider the case $p=q=2$. We have $\delta=0$ and
\[
 \mathcal S^0_{2,2}=\left\{
 \begin{pmatrix}
 a&0&0&0\\
 b&c&0&0\\
 u&0&c&0\\
 0&-u&b&a
 \end{pmatrix}:a,b,c,u\in\bC\right\},
 \qquad
 V^0_{2,2}=\left\{
 \begin{pmatrix}
 0&0&0&0\\
 0&0&0&0\\
 0&0&0&0\\
 x_1&x_2&x_3&x_4
 \end{pmatrix}:x_i\in\bC\right\}.
\]
Hence
\[
 \mathcal S_{2,2}=\left\{
 \begin{pmatrix}
 a&0&0&0\\
 b&c&0&0\\
 u&0&c&0\\
 v_1&v_2&v_3&v_4
 \end{pmatrix}:a,b,c,u,v_i\in\bC\right\},
\]
and therefore
\[
 \Sigma_{2,2}=f_{2,2}+\mathcal S_{2,2}=\left\{
 \begin{pmatrix}
 a&1&0&0\\
 b&c&0&0\\
 u&0&c&1\\
 v_1&v_2&v_3&v_4
 \end{pmatrix}:a,b,c,u,v_i\in\bC\right\}.
\]
We have
\[
 \check U_X=\left\{
 \begin{pmatrix}
 1&0&0&0\\
 \alpha&1&0&0\\
 \beta&\gamma&1&0\\
 \eta&\beta-\alpha\gamma&-\alpha&1
 \end{pmatrix}:\alpha,\beta,\gamma,\eta\in\bC\right\}.
\]

\end{example}

\subsubsection{Proof of the generalized Kostant theorem}
In this section
we establish the generalized Kostant theorem:

\begin{theorem}\label{whittaker-main}
For $1\leq p\leq q$, the action map
\[
 \check U\times\Sigma_{p,q}\longrightarrow\check\fb+f_{p,q},
 \qquad(u,x)\longmapsto\operatorname{Ad}_u(x),
\]
induces an isomorphism
\[
 \check U\times^{\check U_X}\Sigma_{p,q}
       \xrightarrow{\ \sim\ }\check\fb+f_{p,q}.
\]
\end{theorem}
It will be convenient to slightly reformulate
Theorem~\ref{whittaker-main}. Put
\[
 \mathcal P_p=\exp(\check{\mathfrak u}_{2p}\cap\mathfrak p_{2p}).
\]
Multiplication is an isomorphism $\mathcal P_p\times \check U_X\to\check U_{2p}$:
for $u\in\check U_{2p}$, its inverse is
\begin{equation}\label{eq:slice-factorization}
 a=\exp\!\left(\tfrac12\log(u\theta_X(u)^{-1})\right),\qquad k=a^{-1}u.
\end{equation}
Indeed, $\theta_X(a)=a^{-1}$ and $\theta_X(k)=k$. Conversely, a factorization
$u=ak$ with $a\in \mathcal P_p$ and $k\in \check U_X$ forces
$a^2=u\theta_X(u)^{-1}$, so is unique. Exponential and logarithm are
polynomial on these nilpotent and unipotent matrices.
Let $\pi:\check U\to\check U_{2p}$ be the natural upper-left block projection,
and put
\[
 C_{p,q}=\pi^{-1}(\mathcal P_p),\qquad
 \mathfrak c_{p,q}=T_1C_{p,q}.
\]
Applying the factorization~\eqref{eq:slice-factorization} to $\pi(u)$
gives polynomial isomorphisms
\begin{equation}\label{factorization of U}
     C_{p,q}\times \check U_X\simeq \check U,\qquad
 C_{p,q}\simeq\ker\pi\times(\check{\mathfrak u}_{2p}\cap\mathfrak p_{2p}).
\end{equation}
Thus $C_{p,q}$ is an affine transversal, and
$\check{\mathfrak u}=\mathfrak c_{p,q}\oplus\check{\mathfrak u}_X$.

\begin{example}
Consider the case $p=q=2$.
We have 
\[
 \mathfrak c_{2,2}=\check{\mathfrak u}_4\cap\mathfrak p_4
 =\left\{
 \begin{pmatrix}
 0&0&0&0\\
 r&0&0&0\\
 s&0&0&0\\
 0&-s&r&0
 \end{pmatrix}:r,s\in\bC\right\},\ \ \ \ 
C_{2,2}=\mathcal P_2=
 \left\{ 
 \begin{pmatrix}
 1&0&0&0\\
 r&1&0&0\\
 s&0&1&0\\
 0&-s&r&1
 \end{pmatrix}:r,s\in\bC\right\}.
\]

\end{example}

The first factorization in~\eqref{factorization of U} identifies
$\check U\times^{\check U_X}\Sigma_{p,q}\cong C_{p,q}\times\Sigma_{p,q}$.
Hence Theorem~\ref{whittaker-main} will follow from the next proposition.

\begin{proposition}\label{whittaker-main-reformulated}
The action map
\[
 a:C_{p,q}\times\Sigma_{p,q}\longrightarrow\check\fb+f_{p,q},
 \qquad(c,x)\longmapsto\operatorname{Ad}_c(x),
\]
is an isomorphism.
\end{proposition}

The proof will follow from the following observations:
\begin{enumerate}[(i)]
\item The map $a$ is \'etale at $(1,f_{p,q})$.
\item In affine coordinates centered at $(1,f_{p,q})$ and
$f_{p,q}$, the map $a$ intertwines positive linear
$\mathbb G_m$-actions.
\end{enumerate}
These assertions imply that $a$ is an isomorphism by the polynomial
inverse argument given below.

\subsubsection{Proof of (i)}
The following assertion follows by differentiating the action map.

\begin{lemma}\label{general-slice-differential}
The differential of $a$ at $(1,f_{p,q})$ is
\[
 da:\mathfrak c_{p,q}\oplus\mathcal S_{p,q}\longrightarrow\check\fb,
 \qquad (x,y)\longmapsto[x,f_{p,q}]+y.
\]
\end{lemma}
\begin{proof}
A direct computation.
\end{proof}

Therefore, (i) can be reframed as follows.

\begin{lemma}\label{etale-decomposition}
The map $[f_{p,q},-]:\mathfrak c_{p,q}\to\check\fb$ is injective,
and
\[
 \check\fb=[f_{p,q},\mathfrak c_{p,q}]\oplus\mathcal S_{p,q}.
\]
\end{lemma}
\begin{proof}
We deduce the assertion from the next two lemmas, which give the
equality of dimensions and the trivial intersection, respectively.
\end{proof}

\begin{lemma}\label{general-slice-dimensions}
$\dim\check\fb=\dim\mathfrak c_{p,q}+\dim\mathcal S_{p,q}$.
\end{lemma}
\begin{proof}
We have
\[
 \dim\mathfrak c_{p,q}=\frac{N(N-1)}2-p^2,\qquad
 \dim\mathcal S_{p,q}=p^2+\delta+2p=p^2+N.
\]
Their sum is $N(N+1)/2$.
\end{proof}

\Needspace{5\baselineskip}
\begin{lemma}\label{general-slice-intersection}
If $\xi\in\mathfrak c_{p,q}$ and
$[f_{p,q},\xi]\in\mathcal S_{p,q}$, then $\xi=0$.
\end{lemma}
\begin{proof}
We assume first that $\delta>0$.
With $E=\mathbf e_{2p}\varepsilon_1^{\mathrm t}$ as above, write
\[
 f_{p,q}=\begin{pmatrix}f_{p,p}&E\\0&f_\delta\end{pmatrix},\qquad
 \xi=\begin{pmatrix}x&0\\y&z\end{pmatrix},
\]
where $x\in\check{\mathfrak u}_{2p}\cap\mathfrak p_{2p}$,
$y\in M_{\delta,2p}$, and $z\in\check{\mathfrak u}_\delta$.
The hypothesis gives
\[
 [f_{p,q},\xi]=
 \begin{pmatrix}
 [f_{p,p},x]+Ey&Ez-xE\\
 f_\delta y-yf_{p,p}&[f_\delta,z]-yE
 \end{pmatrix}
 =\begin{pmatrix}a&0\\v&b\end{pmatrix},
\]
where $a\in\mathcal S^0_{p,p}$, $b\in\mathcal S_\delta$, and $v$ is
supported in its last row.

We project the first condition $[f_{p,p},x]+Ey=a$ to the
$\theta_X=1$ and $\theta_X=-1$ eigenspaces in $\mathfrak{gl}_{2p}$.
Since $a\in\mathfrak p_{2p}$ and $[f_{p,p},x]\in\check\fg_X$, we obtain
\[
 [f_{p,p},x]+\operatorname{Av}_{\theta_X}(Ey)=0,\qquad
 \tfrac12(Ey-\theta_X(Ey))=a.
\]
In the first equation, $x\in\check{\mathfrak u}_{2p}\cap\mathfrak p_{2p}$ and
$\operatorname{Av}_{\theta_X}(Ey)\in
\operatorname{Av}_{\theta_X}(V^0_{p,p})=\mathfrak v$.
Lemmas~\ref{slice-intersection} and~\ref{adf injectivity}
therefore give $x=0$ and $\operatorname{Av}_{\theta_X}(Ey)=0$.
The projection $\operatorname{Av}_{\theta_X}$ is injective on last-row
matrices: $\theta_X$ carries the last row to the first column, whose
intersection with that row is spanned by the $\theta_X$-fixed matrix
$E_{2p,1}$. Hence $Ey=0$, so the first row of $y$ vanishes.

Now we take up the equation $f_\delta y-yf_{p,p}=v$. It asserts that
the first $\delta-1$ rows of the left hand side vanish.
If $\delta=1$, the vanishing of the first row of $y$ already gives
$y=0$. For $\delta>1$, given a vector $t\in\bC^\delta$, we use the
notation
\[
 t=\begin{pmatrix}t_1\\\vdots\\t_\delta\end{pmatrix},\qquad
 t^+=\begin{pmatrix}t_1\\\vdots\\t_{\delta-1}\end{pmatrix},\qquad
 t^-=\begin{pmatrix}t_2\\\vdots\\t_\delta\end{pmatrix}.
\]
If $y_1,\ldots,y_{2p}\in\bC^\delta$ are the columns of $y$, then
\[
 \begin{aligned}
 f_\delta y&=\begin{pmatrix}
 y_1^-&y_2^-&\cdots&y_{2p}^-\\
 0&0&\cdots&0
 \end{pmatrix},\\[6pt]
 yf_{p,p}&=\begin{pmatrix}
 0&y_1^+&\cdots&y_{p-1}^+&0&y_{p+1}^+&\cdots&y_{2p-1}^+\\
 0&*&\cdots&*&0&*&\cdots&*
 \end{pmatrix},
 \end{aligned}
\]
with empty column ranges omitted when $p=1$. Discarding the final rows
and equating the resulting matrices yields the following system:
\[
 \begin{aligned}
 y_1^-&=0,& y_j^-&=y_{j-1}^+\quad(2\leq j\leq p),\\
 y_{p+1}^-&=0,& y_j^-&=y_{j-1}^+\quad(p+2\leq j\leq2p).
 \end{aligned}
\]
Since the first row of $y$ is zero, $y_1^-=0$ implies $y_1=0$.
Thus $y_1^+=0$, and repeating this argument along the first chain gives
$y_1=\cdots=y_p=0$. The induction begins again from
$y_{p+1}^-=0$ and results in $y_{p+1}=\cdots=y_{2p}=0$.
Thus $y=0$ in all cases. The lower-right block now gives
$[f_\delta,z]=b\in\mathcal S_\delta$. The differential at
$(1,f_\delta)$ of Kostant's isomorphism implies that $z=0$.
This completes the proof when $\delta>0$.

When $\delta=0$, we have
$\xi\in\mathfrak c_{p,p}=\check{\mathfrak u}_{2p}\cap\mathfrak p_{2p}$ and
$[f_{p,p},\xi]\in\mathcal S_{p,p}^0\oplus\mathfrak v$.
Since the commutator belongs to $\check\fg_X$ and
$\mathcal S_{p,p}^0\subseteq\mathfrak p_{2p}$, it belongs to $\mathfrak v$.
Lemma~\ref{slice-intersection} gives $[f_{p,p},\xi]=0$, and
Lemma~\ref{adf injectivity} then gives $\xi=0$.
\end{proof}

\begin{proof}[Proof of (ii) and Proposition~\ref{whittaker-main-reformulated}]
The preceding lemmas establish (i).
Let $\tau_N$ act on $C_{p,q}$ by conjugation and on the affine matrix
spaces by $x\mapsto t^2\operatorname{Ad}_{\tau_N(t)}x$. On the first
$2p$ coordinates, $\tau_N$ differs from $\tau_{2p}$ by a scalar, so it
preserves $\mathcal P_p$ and $C_{p,q}$. In the coordinates
$\ker\pi\times(\check{\mathfrak u}_{2p}\cap\mathfrak p_{2p})$ all weights are
positive: the weights of conjugation on strictly lower triangular
matrix entries are $2(a-b)>0$ for $a>b$. The matrix $f_{p,q}$ is fixed,
while all weights on $\check\fb$ are $2+2(a-b)>0$ for $a\geq b$.
The subspace $\mathcal S_{p,q}$ is stable: its first block and
last-row summand are stable, and, when $\delta>0$, $e_\delta^k$ has
weight $2k+2$ under the affine action. Thus the source and target are
affine spaces with positive linear actions in coordinates centered at
$(1,f_{p,q})$ and $f_{p,q}$, respectively. The map $a$ is
equivariant. This proves (ii).

To conclude, use these affine coordinates on
$C_{p,q}\times\Sigma_{p,q}$ and $\check\fb+f_{p,q}$.
The actions are positive and linear, and $da_{(1,f_{p,q})}$ is an
isomorphism. After an equivariant linear change of coordinates, the
pullback of a coordinate of weight $d$ is itself plus a polynomial in
coordinates of weights strictly less than $d$. Induction on the
weights constructs a polynomial inverse to $a$.
\end{proof}

For comparison with the grading on the Whittaker space, recall
$\tau_{\check K}=\operatorname{diag}(\tau_p,\tau_q)$ and the cocharacter
from~\eqref{Whittaker cocharacter}:
\begin{equation}\label{eq:normal-form-grading}
 \gamma(t)=\tau_{\check K}(t)^2\tau_N(t)^{-1}
 =\operatorname{diag}(\tau_p(t),\tau_q(t))
       \operatorname{diag}(t^qI_p,t^{-p}I_q).
\end{equation}
Conjugation by $\gamma$ preserves $\check U_X$ and $C_{p,q}$,
and $x\mapsto t^2\operatorname{Ad}_{\gamma(t)}x$ preserves
$\Sigma_{p,q}$ and fixes $f_{p,q}$. Indeed, the first $2p$ entries
of $\gamma$ form a cocharacter of $\check G_X$, and its last $\delta$
entries are $\tau_\delta$. The normal-form isomorphism of
Theorem~\ref{whittaker-main} is therefore equivariant for this action
as well. The equivariance required in Theorem~\ref{affine closures} is for
\eqref{eq:normal-form-grading}; $\tau_N$ is used in the contraction argument.

\subsection{Lagrangian correspondence}\label{Lag}
We  consider the case $p=q\geq1$, with the notation of
Section~\ref{generalized Kostant thm}. 
Recall that we use the trace
form to make the equivariant identifications
$\check\fg^*\simeq\check\fg$ and
$\check\fg_X^*\simeq\check\fg_X$. Put $f=f_{p,p}$ and $v_0=\mathbf e_{2p}$.
The restriction of $f$ to $\check{\mathfrak u}$ is $d\psi_{\check K}$.
We use the Whittaker space
\[
T^*_{\psi_{\check K}}(\check G/\check U)=\check G\times^{\check U}(\check\fb+f).
\]
With the convention $(g,z)\sim(gh^{-1},hz)$, the moment map for the
right action on $T^*\check G$ is $-\xi|_{\check{\mathfrak u}}$, so this is the
reduction at $-d\psi_{\check K}$.

We now consider the Hamiltonian $\check G$-variety $\check M=T^*\mathcal X$, where
$\mathcal X=\check G\times^{\check G_X} V\simeq \check G/\check G_X\times V$. Let $x_0=[1,v_0]\in \mathcal X$.
The twisted conormal to the $\check G$-orbit through $(1\check U,x_0)$ in
$(\check G/\check U)\times \mathcal X$ gives a correspondence
\[
\begin{tikzcd}
 &\mathfrak L\ar[ld,"p_1",swap]\ar[rd,"p_2"]&\\
 T^*_{\psi_{\check K}}(\check G/\check U)&&T^*\mathcal X.
\end{tikzcd}
\]
Letting $\check G_{X,v_0}\subseteq \check G_X$ denote the stabilizer of $v_0$ and
$\mu:T^*\mathcal X\to\check\fg^*$ denote the moment map, we have
\[
 F_{x_0}=T^*_{x_0}\mathcal X\cap\mu^{-1}(\check\fb+f),
 \qquad \mathfrak L=\check G\times^{\check U\cap \check G_{X,v_0}}F_{x_0},
\]
with $p_1([g,\eta])=[g,\mu(\eta)]$ and $p_2([g,\eta])=g\eta$.
We observe that $\check U\cap \check G_{X,v_0}=\check U\cap \check G_X=\check U_X$, since $v_0$ is
already $\check U$-invariant. Moreover, $f\in\mathfrak p_{2p}$ implies
$d\psi_{\check K}|_{\check{\mathfrak u}_X}=0$, and hence
$\psi_{\check K}|_{\check U_X}=0$;
this is the stabilizer condition for the twisted conormal construction.

We can write
$T^*_{x_0}\mathcal X\simeq(\check\fg/\check\fg_X)^*\times\{v_0\}\times V^*$.
Using $\omega$, we identify $V^*\simeq V$. The trace form is
$\theta_X$-invariant, and thus identifies
$(\check\fg/\check\fg_X)^*$ with $\mathfrak p_{2p}$. In these
trace-orthogonal coordinates,
\[
 T^*\mathcal X\simeq \check G\times^{\check G_X}(\mathfrak p_{2p}\times V\times V),\qquad
 \mu([g,A,v,w])=\operatorname{Ad}_g(A+\beta(v,w)),
\]
where $i(w)=\beta(v_0,w)$ is the map of the preceding subsection.
Recall the affine space $\mathcal S=\Sigma_{p,p}=f_{}+\mathcal S_{p,p}$
in~\eqref{S}.
Since $i(V)\subset\check\fb\cap\check\fg_X$, the preimage of
$\check\fb+f$ in $T^*_{x_0}\mathcal X$ is
\[
 ((\check\fb+f)\cap\mathfrak p_{2p})\times\{v_0\}\times V
 \simeq\mathcal S,\qquad (A,v_0,w)\longmapsto A+i(w).
\]
Our correspondence becomes
\begin{equation}\label{Lag}
    \begin{tikzcd}
 &\check G\times^{\check U_X}\mathcal S
       \ar[ld,"p_1",swap]\ar[rd,"p_2"]&\\
 \check G\times^{\check U}(\check\fb+f)&&
       \check G\times^{\check G_X}(\mathfrak p_{2p}\times V\times V).
\end{tikzcd}
\end{equation}
The right arrow sends $[g,A+i(w)]$ to $[g,A,v_0,w]$.
Identifying
$\check G\times^{\check U_X}\mathcal S\simeq
\check G\times^{\check U}(\check U\times^{\check U_X}\mathcal S)$, we see that $p_1$ is
induced from the action map
$\bar a:\check U\times^{\check U_X}\mathcal S\to\check\fb+f$, which is
an isomorphism by Theorem~\ref{whittaker-main}, applied with $q=p$.
We conclude that
$p_1$ is an isomorphism.

\begin{lemma}
 $\dim T^*_{\psi_{\check K}}(\check G/\check U)=\dim T^*(\check G\times^{\check G_X} V)$.
\end{lemma}
\begin{proof}
 Equivalently, we claim that $\dim \check G_X=\dim \check U+\dim V$.
 Since $\dim V=\dim\check T$, we are claiming that $\dim \check G_X=\dim \check B$.
 Indeed, both dimensions are $p(2p+1)$.
\end{proof}

The correspondence is Lagrangian. Indeed, the two maps from $\check G\times F_{x_0}$
\[
 (g,\eta)\longmapsto(g,\mu(\eta))\in T^*\check G,\qquad
 (g,\eta)\longmapsto g\eta\in T^*\mathcal X
\]
have equal pullbacks of the Liouville forms: on a tangent vector with
$\check G$-component $ga$, both evaluate as $\langle\mu(\eta),a\rangle$.
Differentiating and descending gives
$p_1^*\omega_{T^*_{\psi_{\check K}}(\check G/\check U)}=p_2^*\omega_{T^*\mathcal X}$; the dimension equality
proves the assertion. Hence $j=p_2p_1^{-1}$ preserves the symplectic
forms and is a morphism of smooth varieties of the same dimension,
so it is \'etale. It follows that $p_2$ is \'etale.

\subsubsection{Big open image}
We shall show 
that $p_2$ in the Lagrangian correspondence~\eqref{Lag} is an open immersion and is surjective up to a closed
subvariety of codimension at least two.

Writing $\check G\times^{\check U_X}\mathcal S\simeq
\check G\times^{\check G_X}(\check G_X\times^{\check U_X}\mathcal S)$, we see that $p_2$ is
induced from a map
\[
 \check G_X\times^{\check U_X}\mathcal S\longrightarrow\mathfrak p_{2p}\times V\times V.
\]
To simplify notation, write
\[
 \mathcal T_0=(f+\check\fb)\cap\mathfrak p_{2p},\qquad
 \mathcal T=\mathcal T_0\times\{v_0\}.
\]
Then
$\mathcal S\simeq\mathcal T\times V$, and under the isomorphism
\[
 \check G_X\times^{\check U_X}(\mathcal T\times V)
 \xrightarrow{\sim}(\check G_X\times^{\check U_X}\mathcal T)\times V,
 \qquad[k,t,w]\longmapsto([k,t],kw),
\]
the map becomes $q\times\operatorname{id}_V$, where
\begin{equation}\label{q}
    q:\check G_X\times^{\check U_X}\mathcal T\longrightarrow\mathfrak p_{2p}\times V,
 \qquad[k,(A,v_0)]\longmapsto(\operatorname{Ad}_kA,kv_0). 
\end{equation}
Because $p_2$ is \'etale, the same is true of $q$, as may be checked
after the smooth surjective base change $\check G\to \check G/\check G_X$. Moreover, to
show that $p_2$ is an open immersion with complement of codimension
at least two, it suffices to show the same of $q$.

Our first step is to show that $q$ is an open immersion.

\begin{lemma}\label{q is open}
 $q$ is an open immersion.
\end{lemma}
\begin{proof}
 Because $q$ is \'etale, it suffices to show that it is injective on
 geometric points. Put $E_0=0$ and
 \[
 E_i=\langle \mathbf e_{2p},\mathbf e_{2p-1},\ldots,\mathbf e_{2p-i+1}\rangle
 \qquad(1\leq i\leq p).
 \]
 For $A\in(\check\fb+f)\cap\mathfrak p_{2p}$, the prescribed
 superdiagonal entries give
 $A^{i-1}v_0\equiv \mathbf e_{2p-i+1}\pmod{E_{i-1}}$.
 Suppose that $[k,(A,v_0)]$ and $[k',(A',v_0)]$ have the same image
 $(x,v)$. Both $k$ and $k'$ carry $E_i$ to
 $\langle v,xv,\ldots,x^{i-1}v\rangle$, with the same normalized
 generators on the successive quotients. Thus $k^{-1}k'$ preserves
 this isotropic flag and acts trivially on its graded pieces.
 Since it is symplectic, it also preserves the dual flag and acts
 trivially on its graded pieces. Hence $k^{-1}k'\in \check U\cap \check G_X=\check U_X$,
 so the two associated-bundle points coincide. The argument works
 over every algebraically closed extension of $\bC$. Thus $q$ is
 universally injective and \'etale, and therefore an open immersion.
\end{proof}

Our next step is to show that $q$ is surjective up to a closed
subvariety of codimension at least two. Since $0\in V$ has
codimension $2p\geq2$, we may replace $V$ by $V\setminus\{0\}$.
But $\check G_X$ acts transitively on $V\setminus\{0\}$. Letting $\check G_{X,v_0}$ denote
the stabilizer of $v_0$ in $\check G_X$, we see that it suffices to show that
the action map
$\check G_{X,v_0}\times((\check\fb+f)\cap\mathfrak p_{2p})\to\mathfrak p_{2p}$
is surjective up to a codimension at least two subvariety.
To this end, define $\ell:\mathfrak p_{2p}\to\bZ$ by setting $\ell(x)$ to be
the dimension of the subspace $V_x$ spanned by
$v:=v_0,xv,x^2v,\ldots$. It is clear that the fibers of $\ell$ are
locally closed subvarieties of $\mathfrak p_{2p}$.

\begin{lemma}\label{isotropic generating space}
 Let $x\in\mathfrak p_{2p}$ and $w\in V$. The subspace spanned by
 $x^iw$ for $i\geq0$ is isotropic.
\end{lemma}
\begin{proof}
 We have
 $\omega(x^iw,x^jw)=\omega(w,x^{i+j}w)=\omega(x^jw,x^iw)$
 by the self-adjointness of $x$. But
 $\omega(x^iw,x^jw)=-\omega(x^jw,x^iw)$, since $\omega$ is
 alternating. So $\omega(x^iw,x^jw)=0$.
\end{proof}

\begin{lemma}\label{regular locus}
 The bound $\ell(x)\leq p$ holds for $x\in\mathfrak p_{2p}$. Moreover,
 it is attained over an open dense subvariety of $\mathfrak p_{2p}$.
\end{lemma}
\begin{proof}
 By Lemma~\ref{isotropic generating space}, the subspace $V_x$ is
 isotropic. So $\ell(x)=\dim V_x\leq\dim V/2=p$. Therefore
 $\ell^{-1}(p)$ is open. It is nonempty because $\ell(f)=p$.
 The density follows from the irreducibility of $\mathfrak p_{2p}$.
\end{proof}

Let $\mathcal U:=\ell^{-1}(p)$ be the dense open subvariety of
Lemma~\ref{regular locus}.

\begin{lemma}\label{isotropic flag variety dimension}
 Let $V$ be a symplectic vector space of dimension $2p$ and let
 $0\leq i\leq p$. Then the variety $Gr_i^\omega$ of
 $i$-dimensional isotropic subspaces $W\subseteq V$ has dimension
 \[
 \dim Gr_i^\omega=i\left(2p-\frac{3i-1}{2}\right).
 \]
\end{lemma}
\begin{proof}
 The case $i=0$ is immediate. Let $Fl_i^\omega$ denote the variety
 of isotropic flags in $V$ of height $i$. There is a smooth
 surjection $Fl_i^\omega\to Gr_i^\omega$ whose fiber over $W$ is
 the variety $Fl(W)$ of complete flags in $W$. Hence
 \[
 \dim Gr_i^\omega=\dim Fl_i^\omega-\dim Fl(W).
 \]
 Recall that $\dim Fl(W)=i(i-1)/2$. The natural map
 $Fl_i^\omega\to Fl_{i-1}^\omega$ has fiber over $W_\bullet$
 isomorphic to $\mathbf P(W_{i-1}^\perp/W_{i-1})$, of dimension
 $2p-2(i-1)-1$. So
 \[
 \dim Fl_i^\omega=\sum_{j=0}^{i-1}(2p-2j-1)=(2p-i)i.
 \]
 Therefore
 \[
 \dim Gr_i^\omega=(2p-i)i-\frac{i(i-1)}2
 =i\left(2p-\frac{3i-1}{2}\right).
 \]
\end{proof}

\begin{lemma}\label{dimension estimate}
 For $1\leq i\leq p$, we have
 \[
 \operatorname{codim}_{\mathfrak p_{2p}}\ell^{-1}(i)=2(p-i).
 \]
 In particular, $\dim \ell^{-1}(i)\leq\dim\mathfrak p_{2p}-2$ for $i<p$.
\end{lemma}
\begin{proof}
 Let $\mathcal Y$ be the variety of isotropic subspaces of $V$ of
 dimension $i$ containing $v$. There is a map
 $\pi:\ell^{-1}(i)\to\mathcal Y$ taking $x$ to $V_x$.
 This map is surjective. Indeed, suppose that $W$ is any such
 subspace. We may choose a symplectic basis
 $v_1,\ldots,v_{2p}=v$ such that $W$ is spanned by
 $v_{2p},\ldots,v_{2p-i+1}$. Let $x$ operate on $V$ by
 $xv_j=v_{j-1}$ for $j\geq2p-i+2$ or $2\leq j\leq i$, and
 $xv_j=0$ otherwise. Then $x\in\mathfrak p_{2p}$,
 $x\in \ell^{-1}(i)$, and $\pi(x)=W$.

 Now we compute $\dim\mathcal Y$. Let $L$ be the span of $v$.
 Then isotropic $i$-planes $W$ containing $L$ correspond to
 isotropic $(i-1)$-planes $W/L$ in the symplectic vector space
 $L^\perp/L$, of dimension $2p-2$.
 By Lemma~\ref{isotropic flag variety dimension},
 \[
 \dim\mathcal Y=(i-1)\left(2p-2-\frac{3(i-1)-1}{2}\right)
 =\frac{(i-1)(4p-3i)}2.
 \]
 Fix $W\in\mathcal Y$ and consider the fiber $\pi^{-1}(W)$.
 Since $x$ preserves $W$ and is self-adjoint, it operates on $W$
 and $W^\perp/W$. The symplectic form descends to $W^\perp/W$,
 and the action of $x$ is self-adjoint with respect to it.
 Let $\mathfrak p_{W^\perp/W}$ be the space of these endomorphisms.
 We get a map
 \[
 \eta:\pi^{-1}(W)\longrightarrow
 \mathfrak{gl}_W\times\mathfrak p_{W^\perp/W}.
 \]
 Since $\dim(W^\perp/W)=2(p-i)$, the base has dimension
 \[
 \dim\mathfrak{gl}_W+\dim\mathfrak p_{W^\perp/W}
 =i^2+\dim\bigwedge^2(W^\perp/W)^*
 =i^2+(p-i)(2(p-i)-1).
 \]
 Fix $x_1\in\mathfrak{gl}_W$ and
 $x_2\in\mathfrak p_{W^\perp/W}$. The image of $\eta$ consists
 of all such pairs for which $v$ generates $W$ under $x_1$.
 This is an open condition on $x_1$, so the image is open.
 Now consider a nonempty fiber $\eta^{-1}(x_1,x_2)$.
 Let $P_W\subseteq \check G$ be the stabilizer of the flag
 $W\subseteq W^\perp$, let $U_W$ be its unipotent radical, and
 put $\mathfrak u_W=\operatorname{Lie}(U_W)$.
 Then $\mathfrak u_W\cap\mathfrak p_{2p}$ operates on the fiber by
 translation, making it a $\mathfrak u_W\cap\mathfrak p_{2p}$-torsor.
 In a symplectic decomposition $V=W\oplus E\oplus W^*$, its free
 blocks are an arbitrary $W\otimes E$ block and an alternating
 $W\otimes W$ block. Thus
 \[
 d:=\dim(\mathfrak u_W\cap\mathfrak p_{2p})
 =2i(p-i)+\frac{i(i-1)}2.
 \]
 Putting it all together,
 \begin{align*}
 \dim \ell^{-1}(i)
 &=\frac{(i-1)(4p-3i)}2+i^2+(p-i)(2(p-i)-1)+d\\
 &=2p^2-3p+2i.
 \end{align*}
 By comparison, $\dim\mathfrak p_{2p}=\dim\bigwedge^2V^*=2p^2-p$.
 So the desired codimension is
 \[
 \operatorname{codim}_{\mathfrak p_{2p}}\ell^{-1}(i)
 =(2p^2-p)-(2p^2-3p+2i)=2(p-i),
 \]
 which is at least two when $i<p$.
\end{proof}

\begin{corollary}
 The complement of the open subvariety $\mathcal U=\ell^{-1}(p)$
 in $\mathfrak p_{2p}$ has codimension at least two.
\end{corollary}
\begin{proof}
 The complement $\mathfrak p_{2p}\setminus\mathcal U$ is the union
 of the subvarieties $\ell^{-1}(i)$ for $1\leq i<p$, so the claim
 follows from Lemma~\ref{dimension estimate}. For $p=1$ the
 complement is empty.
\end{proof}

\begin{lemma}
 Every $x\in\mathcal U$ is conjugate under $\check G_{X,v_0}$ to an element
 of $(\check\fb+f)\cap\mathfrak p_{2p}$.
\end{lemma}
\begin{proof}
 Since $x\in\mathcal U$, the vectors
 $v_{2p}:=v_0,v_{2p-1}:=xv_0,\ldots,v_{p+1}:=x^{p-1}v_0$
 span a Lagrangian subspace $W$.
 By Lemma~\ref{isotropic generating space} and the maximality of
 $W$, we have $x^pv_0\in W$. Hence $W$ is $x$-stable and $x$
 operates on $V/W$. The form $\omega$ induces a perfect pairing
 $W\times V/W\to\bC$.
 Choose a Lagrangian complement to $W$ and its basis
 $v_1,\ldots,v_p$ such that
 $\omega(v_i,v_{2p+1-j})=\delta_{ij}$.
 Then there exists a unique $k\in \check G_X$ such that $k\mathbf e_i=v_i$.
 Since $k\mathbf e_{2p}=v_0=\mathbf e_{2p}$, we have $k\in \check G_{X,v_0}$.

 In this basis the lower-right block $D$ of $k^{-1}xk$ has ones
 on its superdiagonal and zeros above it. The upper-right block
 is zero because $W$ is stable. Self-adjointness makes the
 upper-left block $J_pD^tJ_p$, with the same superdiagonal
 pattern. All remaining entries lie on or below the diagonal.
 Thus $k^{-1}xk\in\check\fb+f$.
\end{proof}

\begin{corollary}\label{action map codimension}
 The image of the action map
 $q:\check G_X\times^{\check U_X}\mathcal T\to\mathfrak p_{2p}\times V$ is
 \[
 \mathcal U_{\mathrm{cyc}}=\{(x,v):v,xv,\ldots,x^{p-1}v
                         \text{ are linearly independent}\}.
 \]
 The complement of its image has codimension two.
\end{corollary}
\begin{proof}
 The prescribed superdiagonal entries give
 $\operatorname{im}q\subseteq\mathcal U_{\mathrm{cyc}}$. The reverse inclusion
 follows from the preceding lemma and the transitivity of $\check G_X$
 on $V\setminus\{0\}$. Over nonzero $v$, the stratum of cyclic
 dimension $i$ has codimension $2(p-i)$ by
 Lemma~\ref{dimension estimate}; the locus $v=0$ has codimension
 $2p$. For $p\geq2$, the nonempty stratum $i=p-1$ has codimension
 two. For $p=1$, the complement is $\mathfrak p_{2p}\times\{0\}$,
 again of codimension two.
\end{proof}

Putting things together, we get:
\begin{proposition}\label{big open}
The  action map
 $q:\check G_X\times^{\check U_X}\mathcal T\to\mathfrak p_{2p}\times V$ in~\eqref{q}, hence also the $p_2$ in the Lagrangian correspondence~\eqref{Lag},
is an open immersion and
is surjective up to a closed
subvariety of codimension at least two.
\end{proposition}

\subsection{Proof of Theorem~\ref{affine closures}}
When $p=q$, Proposition \ref{big open} implies that 
the Lagrangian correspondence gives an open embedding
\[
 j=p_2p_1^{-1}:T^*_{\psi_{\check K}}(\check G/\check U)
       \longrightarrow\check M
\]
whose complement has codimension two. The construction is
$\check G\times\mathbb G_m$-equivariant for the grading fixed in
Section~3.4. Hartogs' principle therefore gives the desired isomorphism
of affine closures.

Now assume $p<q$, and put $\delta=q-p>0$. Recall that
$\check G=GL_N$, $N=p+q$, $\check G_X=\Sp_{2p}$, and
$\check U_X=\check G_X\cap\check U$. The same slices are
\[
 \mathcal T_0=(f_{p,p}+\mathfrak b_{2p})\cap\mathfrak p_{2p},
 \qquad \mathcal T=\mathcal T_0\times\{v_0\},\qquad v_0=\mathbf e_{2p},
\]
and we use $\mathcal S_\delta=Z_{\mathfrak{gl}_\delta}(e_\delta)$.
By Theorem~\ref{whittaker-main},
\[
 T^*_{\psi_{\check K}}(\check G/\check U)
 \simeq\check G\times^{\check U_X}\Sigma_{p,q},
 \qquad \Sigma_{p,q}=f_{p,q}+\mathcal S_{p,q}.
\]
On the other hand, the block description~\eqref{slodowy blocks} gives
a natural inclusion
\[
 \begin{aligned}
 \Sigma_{p,q}&\simeq\mathcal T\times V^*\times\mathcal S_\delta
       \lhook\joinrel\longrightarrow\mathcal Z_\delta,\\
 \mathcal Z_\delta&\simeq
       (\mathfrak p_{2p}\times V)\times V^*\times\mathcal S_\delta.
 \end{aligned}
\]
It induces a map
\[
 \begin{aligned}
 j:T^*_{\psi_{\check K}}(\check G/\check U)
 &\simeq\check G\times^{\check U_X}\Sigma_{p,q}
 &\simeq\check G\times^{\check G_X}
       (\check G_X\times^{\check U_X}\Sigma_{p,q})
       \longrightarrow\check G\times^{\check G_X}\mathcal Z_\delta
       \simeq\check M.
 \end{aligned}
\]
We claim that $j$ is a $\check G\times\mathbb G_m$-equivariant open
embedding with codimension two complement. Hartogs' principle then gives
\[
 \overline{T^*_{\psi_{\check K}}(\check G/\check U)}\cong\check M.
\]
Proof of the claim.
It suffices to consider the action map
\begin{equation}\label{action}
 \check G_X\times^{\check U_X}\Sigma_{p,q}\longrightarrow\mathcal Z_\delta.
\end{equation}
The $\check G_X$-action on $\mathcal Z_\delta$ is
\[
 h\cdot(A,b,c,s)=(\operatorname{Ad}_hA,hb,ch^{-1},s).
\]
Consequently, there is a product identification
\[
 \begin{aligned}
 \check G_X\times^{\check U_X}\Sigma_{p,q}
 &\xrightarrow{\ \sim\ }
 (\check G_X\times^{\check U_X}\mathcal T)\times V^*\times\mathcal S_\delta,\\
 [h,(A,v_0,c,s)]
 &\longmapsto\bigl([h,(A,v_0)],ch^{-1},s\bigr).
 \end{aligned}
\]
Under this identification, the action map~\eqref{action} is
\[
 \begin{aligned}
 q\times\operatorname{id}_{V^*}\times\operatorname{id}_{\mathcal S_\delta}:
 (\check G_X\times^{\check U_X}\mathcal T)\times V^*\times\mathcal S_\delta
 \longrightarrow(\mathfrak p_{2p}\times V)\times V^*\times\mathcal S_\delta,
 \end{aligned}
\]
where
$q$ is the map in
Proposition \ref{big open}. Since 
$q$ is an open embedding with codimension two complement, the same is
therefore true for~\eqref{action}, and hence for $j$. This proves the claim.

\section{Main results}

    \subsection{Ext algebras}
    Let $\IC_{\mathcal O(\check G)}\in\text{Ind}(D^b_{\mathfrak L^+G}(\Gr))$ be the regular perverse sheaf on 
    $\Gr$: the algebra ind-object  corresponding to 
the regular representation $\mathcal O(\check G)$ of $\check G$ under the geometric Satake equivalence~\eqref{GS}.
The de-equivariantized Ext
algebras
for the complex, real and relative  Satake categories  
are the dg-algebras
$A=\mathrm{RHom}_{\text{Ind}(D^b_{\frak L^+G}(\Gr_G))}(\IC^0,\IC^0\star\IC_{\mathcal O(\check G)})$, 
$A_\bR=\mathrm{RHom}_{\text{Ind}(D^b_{\frak L^+G_\bR}(\Gr_\bR))}(\IC_\bR^0,\IC_\bR^0\star\IC_{\mathcal O(\check G)})$, and 
$A_X=\mathrm{RHom}_{\text{Ind}(D^b_{\mathfrak L^+G}(\mathfrak LX))}(\IC_X^0,\IC_X^0\star\IC_{\mathcal O(\check G)})$, 
respectively.
Note that  we have a residual $\check G$-action
on $A$ (resp. $A_\bR$, $A_X$)
coming from the right $\check G$-action on $\mathcal O(\check G)$ and a natural 
$\check G$-equivariant dg algebra maps 
$A\to A_\bR$ and $A\to A_X$ induced from the Hecke action.
It follows from derived Satake equivalence for $G$ \cite{BF} that $A$ is formal 
, that is, it is isomorphic to its cohomology algebra $H^*(A)$ with trivial differentials. Moreover, there is a $\check G\times\mathbb G_m$-equivariant isomorphism 
$H^*(A)\cong\mathcal O(\check\fg^*)$
where $\check G\times\mathbb G_m$-acts on $\check\fg^*$ by 
$(g,t)(v)=t^2\text{Ad}_g(v)$.

\begin{theorem}\label{Ext}
 (1) 
 There is a $\check G$-equivariant isomorphism $A_\bR\cong A_X$
 of dg-algebras.
 Moreover,
dg-algebra $A_\bR\cong A_X $ is formal, that is,
it is isomorphic to its cohomology algebra 
$H^*(A_\bR)\cong H^*(A_X)$ with trivial differentials.

(2)
There is $\check G\times\mathbb G_m$-equivariant isomorphism 
$H^*(A_\bR)\cong 
H^*(A_X)
\cong\mathcal O(\check M)$
of $\mathcal O(\check\fg^*)$-algebras.
\end{theorem}
\begin{proof}
Part (1)  is  proved in \cite[Theorem 13.4]{CN}
using the pointwise purity of $\IC^\lambda_X$
in \cite[Theorem 81]{CY}.
Proof of (2).  
According to \cite[(3.1.5)]{Gi}
or \cite [Lemma 4.2]{C}\footnote{In \emph{loc. cit.}, the setting is for  twisted loop groups but the same argument applies to the un-twisted setting.}
we have  a $\check G$-equivariant isomorphism of dg-algebras
\[H^*(A_X)\cong
\mathrm{Ext}^*
_{\mathrm{Ind}(D^b_{\frak L^+G}(\frak LX))}(\IC_X^0,\IC_X^0\star\IC_{\mathcal O(\check G)})\cong\mathrm{H}^*_K(\Gr_K,\iota_K^!\IC_{\mathcal O(\check G)})\]
where $\iota_K:\Gr_K\to\Gr_G$ is natural inclusion map.
On the other hand, \cite[Theorem 1.5.2]{M} and
\cite[Corollary 3.1.6]{Gi}, together with
Proposition~\ref{Wh=J}, identify this algebra with
$\mathcal O(T^*_{\psi_{\check K}}(\check G/\check U))$ over
$\mathcal O(\check\fg^*)$.
Now the claim follows from Theorem \ref{affine closures}.

\end{proof}

Introduce the non-equivaraint dg extension algebra
$A'_\bR=\mathrm{RHom}_{\text{Ind}(D^b_{}(\Gr_\bR))}(\IC_\bR^0,\IC_\bR^0\star\IC_{\mathcal O(\check G)})$ for $D^b_{}(\Gr_\bR)$.

\begin{theorem}\label{Non-equ Ext}
(1) The dg algebra $A'_\bR$ is formal. 

(2) There is a 
$\mathbb G_m$-equivaraint isomorphism 
$H^*(A_\bR')\cong\mathcal O(\chi_{\check M}^{-1}(0))
\cong\mathcal O(GL_{p+q}\times^{\Sp_{2p}}\mathcal N'_{\Sp_{2p}})$
    
\end{theorem}
\begin{proof}
    (1) 
     Note that there are isomorphisms of dg algebras
\[A_\bR' \cong \mathrm{RHom}(\IC_\bR^0, \psi(\IC_{\mathcal O(\check G)}))\cong \mathrm{RHom}_{}(\IC_\bR^0, p_*\IC_{\mathcal O(\check G)})\cong  \mathrm{RHom}_{}(p^*\IC_\bR^0, \IC_{\mathcal O(\check G)})\cong \]
\[\stackrel{(a)}\cong\mathrm{RHom}_{}(\iota_{K*}\underline{\bC}_{\Gr_K}, \IC_{\mathcal O(\check G)})\cong  R\Gamma_{}(\Gr_K, \iota_K^! \IC_{\mathcal O(\check G)}).
  \]
  Here $p:\Gr\to \Gr_\bR$
  is the specialization map 
  such that $p_*\cong\psi:D^b_c(\Gr)\to D^b_c(\Gr_\bR)$
  (see, e.g., \cite[Lemma 4.8]{CMNO})
  and  $(a)$ follows from  the observation that $p^{-1}(t^0) = \Gr_K$. 
  Since $\iota^!_K\IC_{\mathcal O(\check G)}$ 
is $!$-pure with respect to the Iwahori stratification
on $\Gr_K$, by \cite[Lemma 3.1.5]{BY}, the cohomology group $H^i(\Gr_K,\iota^!_K\IC_{\mathcal O(\check G)})$
is pure of weight $i$, and the formality of $R\Gamma_{}(\Gr_K, \iota_K^! \IC_{\mathcal O(\check G)})$
 follows from the general criterion in \cite{BF}.

 (2) Taking $\check G$-invariants in the proof of Theorem~\ref{Ext}(2)
gives $H_K^*(\on{pt})\cong\mathcal O(\check M)^{\check G}$ as graded
augmented algebras: the right regular invariants of $\mathcal O(\check G)$
are the constants.
By \cite[Proposition 2.1.1 and Remark 2.1.2]{M}, the complex
$\iota_K^!\IC_{\mathcal O(\check G)}$ is equivariantly formal. Hence we obtain isomorphisms
\[H^*(A_\bR')\cong H^*(\Gr_K,\iota_K^!\IC_{\mathcal O(\check G)})\cong H^*_K(\Gr_K,\iota_K^!\IC_{\mathcal O(\check G)})\otimes_{H^*_K(\on{pt})}\bC\cong \]
\[\cong\mathcal O(\check M)\otimes_{\mathcal O(\check M\sslash\check G)}\bC\cong\mathcal O(\chi_{\check M}^{-1}(0))\cong\mathcal O(GL_{p+q}\times^{\Sp_{2p}}\mathcal N'_{\Sp_{2p}}).\]

\end{proof}

\subsection{Relative duality for linear periods and derived Satake for $U(p,q)$}\label{sect:shearing}
Let $\mathrm{QCoh}^{\check G}(\check M)$ be the 
dg category of $\check G$-quasi-coherent complexes on the dual variety $\check M$.
Let $\mathrm{QCoh}^{\check G}(\check M)^\renfatslash$ be sheared category 
with respect to the $\mathbb G_m$-action 
on $\check M$.
According to Proposition \ref{induced}, we have 
\begin{equation}\label{M=V}
    \mathrm{QCoh}^{\check G}(\check M)^\renfatslash\cong\mathrm{QCoh}^{\Sp_{2p}}(\wedge^2\mathbb A^{2p}[2]\times T^*\mathbb A^{2p}[q-p+1]\times\frak{gl}^*_{q-p}[2]\sslash GL_{q-p}).
\end{equation}
Here for a finite dimensional vector space $V$ and an integer $d$,
we denote by $V[d]$ the dg scheme with ring of functions the dg algebra $\Sym^*((\Pi^dV)^*[-d])$ with trivial differentials, where  $\Pi^dV$ is 
the super vector space $V$ with parity 
equal to $d$ mod $2$.\footnote{In particular, if $d$ is odd, then $\Pi^dV$
is an odd vector space and  
$\Sym^*((\Pi^dV)^*[-d])$ is 
isomorphic to the symmetric algebra with generators in $V^*$ in degree $d$.}

Denote by 
$\mathrm{QCoh}^{\check G}(\check M)^{\renfatslash}_{\mathcal N_{\check\fg^*}}$
the full subcategory consisting of 
complexes set theoretically supported on the pre-image of the 
nilpotent cone $\mu^{-1}(\mathcal N_{\check\fg^*})$ under the moment map
$\mu:\check M\to\check\fg^*$.
It follows from
Lemma \ref{exotic nilpotent cone} that the above equivalence restricts to an equivalence
\[\mathrm{QCoh}^{\check G}(\check M)^{\renfatslash}_{\mathcal N_{\check\fg^*}}\cong\mathrm{QCoh}^{\Sp_{2p}}(\wedge^2\mathbb A^{2p}[2]\times T^*\mathbb A^{2p}[q-p+1]\times\frak{gl}^*_{q-p}[2]\sslash GL_{q-p})_{\mathcal N_{\Sp_{2p}}'}\]
where the right hand side is the full-subcategory consisting of complexes supported on

\[\mathcal N_{\Sp_{2p}}'\times\{0\}\subset (\wedge^2\mathbb A^{2p}[2]\times T^*\mathbb A^{2p}[q-p+1])\times(\frak{gl}_{q-p}^*[2]\sslash GL_{q-p})\] 
where $\mathcal N_{\Sp_{2p}}'$ is the exotic nilpotent cone 
for $\Sp_{2p}$
in Section \ref{invariant quotients}.

Denote by $\mathrm{Perf}^{\check G}(\check M)^\renfatslash\cong\mathrm{Perf}^{\Sp_{2p}}(\wedge^2\mathbb A^{2p}[2]\times T^*\mathbb A^{2p}[q-p+1]\times\frak{gl}^*_{q-p}[2]\sslash GL_{q-p})$
and 
$\mathrm{Perf}^{\check G}(\check M)^\renfatslash_{\mathcal N_{\check\fg^*}}\cong\mathrm{Perf}^{\Sp_{2p}}(\wedge^2\mathbb A^{2p}[2]\times T^*\mathbb A^{2p}[q-p+1]\times\frak{gl}^*_{q-p}[2]\sslash GL_{q-p})_{\mathcal N_{\Sp_{2p}}'}$
be full subcategory of compact objects consisting of perfect complexes.

\begin{theorem}\label{main}
(1)    There is an equivalence of dg categories
\[\begin{gathered}
\on{Ind}(\Phi^b):\mathrm{Ind}(D^b_{\frak L^+U(p,q)}(\Gr_{U(p,q)}))\cong\mathrm{QCoh}^{\check G}(\check M)^\renfatslash\\
\cong \mathrm{QCoh}^{\Sp_{2p}}(\wedge^2\mathbb A^{2p}[2]\times T^*\mathbb A^{2p}[q-p+1]\times\frak{gl}^*_{q-p}[2]\sslash GL_{q-p})
\end{gathered}\]
compatible with the action of 
derived Satake equivalence 
$\on{Ind}(D^b_{\frak L^+G}(\Gr_{G}))\cong\mathrm{QCoh}^{\check G}(\check\fg^*[2])$ on both sides.
Moreover, it restricts to an equivalence 
\[\begin{gathered}
\Phi^b:D^b_{\frak L^+U(p,q)}(\Gr_{U(p,q)})\cong\mathrm{Perf}^{\check G}(\check M)^\renfatslash\\
\cong \mathrm{Perf}^{\Sp_{2p}}(\wedge^2\mathbb A^{2p}[2]\times T^*\mathbb A^{2p}[q-p+1]\times\frak{gl}^*_{q-p}[2]\sslash GL_{q-p})
\end{gathered}\]
on compact objects.

(2)    There is an equivalence of dg categories
\[\begin{gathered}
\Phi:D_{\frak L^+U(p,q)}(\Gr_{U(p,q)})\cong\mathrm{QCoh}^{\check G}(\check M)^\renfatslash_{\mathcal N_{\check\fg^*}}\\
\cong \mathrm{QCoh}^{\Sp_{2p}}(\wedge^2\mathbb A^{2p}[2]\times T^*\mathbb A^{2p}[q-p+1]\times\frak{gl}^*_{q-p}[2]\sslash GL_{q-p})_{\mathcal N_{\Sp_{2p}}'}
\end{gathered}\]
compatible with the action of 
derived Satake equivalence 
$D_{\frak L^+G}(\Gr_{G})\cong\mathrm{QCoh}^{\check G}(\check\fg^*[2])_{\mathcal N_{\check\fg^*}}$ on both sides.
Moreover, it restricts to an equivalence 
\[\begin{gathered}
\Phi^c:D_{\frak L^+U(p,q)}(\Gr_{U(p,q)})^c\cong\mathrm{Perf}^{\check G}(\check M)^\renfatslash_{\mathcal N_{\check\fg^*}}\\
\cong \mathrm{Perf}^{\Sp_{2p}}(\wedge^2\mathbb A^{2p}[2]\times T^*\mathbb A^{2p}[q-p+1]\times\frak{gl}^*_{q-p}[2]\sslash GL_{q-p})_{\mathcal N_{\Sp_{2p}}'}
\end{gathered}\]
on compact objects.

The same is true when $\Gr_{U(p,q)}$
and $\mathfrak L^+U(p,q)$ are replaced by 
$\frak LX$ and $\frak L^+G$.
\end{theorem}
\begin{proof}
    The same proof as in \cite[Theorem 5.5.]{CMNO} or \cite[Theorem 4.8]{C} in view of Proposition \ref{generation} and 
    the Theorem \ref{Ext}.

\end{proof}

\subsection{Coherent sheaves on exotic nilpotent cones}
Consider the full-subcategory 
\[D^b_{\{\frak L^+U(p,q)\}}(\Gr_{U(p,q)})\subset 
D^b_{}(\Gr_{U(p,q)})\]
consisting of complexes that are constructible with respect to the $\frak L^+U(p,q)$-orbits stratification
on $\Gr_{U(p,q)}$.
Let
\[
\mathrm{QCoh}
^{GL_{p+q}}(GL_{p+q}\times^{\Sp_{2p}}\mathcal N'_{\Sp_{2p}})^\renfatslash
\]
be the sheared dg-category of equivariant quasi-coherent complexes on
the null-cone
$\chi_{\check M}^{-1}(0)\cong GL_{p+q}\times^{\Sp_{2p}}\mathcal N'_{\Sp_{2p}}$.
It has a natural pushforward functor to
$\mathrm{QCoh}^{\check G}(\check M)^\renfatslash$.

Let $\mathrm{Perf}
^{GL_{p+q}}(GL_{p+q}\times^{\Sp_{2p}}\mathcal N'_{\Sp_{2p}})^\renfatslash\subset \mathrm{QCoh}
^{GL_{p+q}}(GL_{p+q}\times^{\Sp_{2p}}\mathcal N'_{\Sp_{2p}})^\renfatslash$ be the subcategory of compact objects.

\begin{theorem}\label{exotic-coherent}
    There is an equivalence of dg categories
\[\mathrm{Ind}(D^b_{\{\frak L^+U(p,q)\}}(\Gr_{U(p,q)}))\cong\mathrm{QCoh}
^{GL_{p+q}}(GL_{p+q}\times^{\Sp_{2p}}\mathcal N'_{\Sp_{2p}})^\renfatslash\]
which restricts to an equivalence 
\[D^b_{\{\frak L^+U(p,q)\}}(\Gr_{U(p,q)})\cong\mathrm{Perf}
^{GL_{p+q}}(GL_{p+q}\times^{\Sp_{2p}}\mathcal N'_{\Sp_{2p}})^\renfatslash\]
on compact objects.
\end{theorem}
\begin{proof}
 The same proof as in \cite[Theorem 5.5.]{CMNO} or \cite[Theorem 4.8]{C} in view of Proposition \ref{generation} and 
    the Theorem \ref{Non-equ Ext}.
    
\end{proof}
\begin{remark}
    Unlike Theorem \ref{main}, we do not know an analogy of Theorem \ref{exotic-coherent} in the setting of $\mathfrak LX$.
\end{remark}

\subsection{Abelian equivalences}\label{sect:abelian}
Put $\delta=q-p$, $H=\check G_X=\Sp_{2p}$, $V=\mathbb A^{2p}$, and
$W=T^*V=V\oplus V^*$. Consider the dg algebras with zero differential
\begin{align*}
\mathfrak A_{p,q}
 &=\Sym^*\left(\wedge^2V[1]\oplus(\Pi^{\delta+1}W)[\delta]
             \oplus\bigoplus_{i=1}^{\delta}\mathbb A^1[2i-1]\right),\\
\mathfrak B_{p,q}
 &=\Sym^*\left(\wedge^2V[-2]\oplus(\Pi^{\delta+1}W)[-\delta-1]
             \oplus\bigoplus_{i=1}^{\delta}\mathbb A^1[-2i]\right).
\end{align*}
Write $\mathfrak A=\mathfrak A_{p,q}$ and $\mathfrak B=\mathfrak B_{p,q}$. Our convention
is that $E[k]$ lies in cohomological degree $-k$, and the symmetric
algebras use the sign $(-1)^{ij+\epsilon\eta}$ for cohomological
degrees $i,j$ and internal parities $\epsilon,\eta$. All modules are
left modules. Their internal parity is prescribed by the central
element $(-I)^{\delta+1}\in H$; thus an independent parity shift does
not give additional simple objects.

The algebra $\mathfrak A$ is finite-dimensional and nonpositive, while $\mathfrak B$ is
polynomial with generators in strictly positive degrees and
$\mathfrak B^0=\mathbb C$. Let $D^H_{fd}(\mathfrak A\on{-mod})$ denote the category of
equivariant dg modules with finite total cohomology, and let
$D^H_{perf}(\mathfrak B\on{-mod})$ be the thick subcategory generated by
$\mathfrak B\otimes U$ for finite-dimensional $H$-representations $U$.
Using the $H$-equivariant self-dualities of $\wedge^2V$ and $W$,
Koszul duality gives
\[
\kappa:D^H_{perf}(\mathfrak B\on{-mod})\xrightarrow{\sim}D^H_{fd}(\mathfrak A\on{-mod}),
\qquad M\longmapsto\mathbb C\overset{\mathbf L}{\otimes}_{\mathfrak B} M,
\]
where the derived fiber carries its Koszul $\mathfrak A$-action; see
\cite[Section 3.8]{BFGT}. Equivalently, $\kappa$ is the inverse of
$N\mapsto R\Hom_{\mathfrak A}(\mathbb C,N)$ under
$\mathfrak B\simeq R\operatorname{End}_{\mathfrak A}(\mathbb C)^{\mathrm{op}}$.

The standard $t$-structure on $D^H_{fd}(\mathfrak A\on{-mod})$ is bounded, with
heart the finite-dimensional $H$-equivariant $H^0(\mathfrak A)$-modules. Hence
\[
D^H_{fd}(\mathfrak A\on{-mod})^\heartsuit\cong
\begin{cases}
\wedge^*W\on{-mod}(\mathrm{Rep}(H)),&p=q,\\
\mathrm{Rep}(H),&p<q.
\end{cases}
\]
Recall from Theorem~\ref{main} the derived Satake equivalence
\[
\begin{aligned}
\Phi^b:D^b_{\frak L^+U(p,q)}(\Gr_{U(p,q)})
&\xrightarrow{\sim}\mathrm{Perf}^{H}\left(
\wedge^2V[2]\times W[\delta+1]\times
\frak{gl}_{\delta}^*[2]\sslash GL_{\delta}\right)
&\xrightarrow{\sim}D^H_{perf}(\mathfrak B\on{-mod})
\end{aligned}
\]

\Needspace{16\baselineskip}
\begin{theorem}\label{abelian equ}
The composed equivalence
\[
\kappa\circ\Phi^b:D^b_{\frak L^+U(p,q)}(\Gr_{U(p,q)})
\xrightarrow{\sim}D^{\Sp_{2p}}_{fd}(\mathfrak A_{p,q}\on{-mod})
\]
is $t$-exact for the perverse $t$-structure on the left and the standard
$t$-structure on the right. It induces an equivalence of abelian categories
\[
\mathrm{Perv}_{\frak L^+U(p,q)}(\Gr_{U(p,q)})\cong
\begin{cases}
\wedge^*(T^*\mathbb A^{2p})\on{-mod}(\mathrm{Rep}(\Sp_{2p})),&p=q,\\
\mathrm{Rep}(\Sp_{2p}),&p<q.
\end{cases}
\]
Here modules and representations are finite-dimensional. The same holds
for $\frak LX$ with its $\frak L^+G$-action.
\end{theorem}

\begin{proof}
Let $\lambda\in\Lambda_A^+$ be a dominant real coweight, and let
$V^\lambda\in\mathrm{Rep}(\Sp_{2p})$ and $\IC^\lambda_\bR$ be the
corresponding irreducible representation and IC complex. It follows
from the construction that
\[
\Phi^b(\IC^\lambda_\bR)\cong\mathfrak B_{p,q}\otimes V^\lambda,
\qquad
\kappa(\mathfrak B_{p,q}\otimes V^\lambda)\cong V^\lambda.
\]
Hence $\kappa\circ\Phi^b(\IC^\lambda_\bR)\cong V^\lambda$ lies in the
standard heart. According to Proposition~\ref{generation}(1), the
simple perverse sheaves are the $\IC^\lambda_\bR$. Since every object
of the perverse heart has finite length, $\kappa\circ\Phi^b$ sends
this heart into the standard heart. Both $t$-structures are bounded,
so $\kappa\circ\Phi^b$ is $t$-exact and induces an equivalence
\[\mathrm{Perv}_{\frak L^+U(p,q)}(\Gr_{U(p,q)})
\cong D^{\Sp_{2p}}_{fd}(\mathfrak A_{p,q}\on{-mod})^\heartsuit.
\]
The theorem follows from the description of the heart above. The
assertion for $\frak LX$ follows from Proposition~\ref{real-symmetric}.
\end{proof}

\end{document}